%% file: decomp.tex
\documentclass[10pt]{amsart}
\usepackage{amssymb,amsmath}
\usepackage{textcomp}
\usepackage{url}
\usepackage{booktabs}
\usepackage{array}
\usepackage{tikz}
\usepackage{nicefrac}
\usetikzlibrary{shapes.misc}

\newcommand{\Q}{\mathbf{Q}}
\newcommand{\C}{\mathbf{C}}

\newcommand{\Z}{\mathbf{Z}}
\newcommand{\Fp}{\mathbf{F}_p}
\newcommand{\Fq}{\mathbf{F}_q}

\newcommand{\af}{\mathfrak{a}}
\newcommand{\lf}{\mathfrak{l}}
\newcommand{\pf}{\mathfrak{p}}
\newcommand{\ff}{\mathfrak{f}}

\newcommand{\inkron}[2]{\genfrac {(}{)}{0.9pt}{}{#1}{#2}}
\newcommand{\M}{\textsf{M}}
\newcommand{\PD}{\mathcal{P}_D}
\newcommand{\cD}{\mathcal{D}}
\newcommand{\idx}[2]{[#1\hspace{1.5pt}\text{\rm :}\hspace{2pt}#2]}

\newcommand{\bmax}{b_{\max}}
\newcommand{\ihat}{\hat{\imath}}

\newcommand{\Tf}{T_{\rm find}}
\newcommand{\Te}{T_{\rm enum}}
\newcommand{\Tb}{T_{\rm build}}
\newcommand{\Tc}{T_{\rm crt}}
\newcommand{\Tr}{T_{\rm root}}
\newcommand{\Tp}{T_{\rm poly}}
\newcommand{\Tt}{T_{\rm tot}}

\renewcommand{\vec}[1]{\boldsymbol{#1}}
\renewcommand{\O}{\mathcal{O}}

\def\cl{\operatorname{cl}}
\def\Gal{\operatorname{Gal}}
\def\End{\operatorname{End}}
\def\lht{\operatorname{ht}}
\def\EllO{\text{\rm Ell}_\mathcal{O}}

\newtheorem{lemma}{Lemma}
\newtheorem{proposition}{Proposition}
\newtheorem*{proposition*}{Proposition}
\newtheorem{claim}{Heuristic Claim}

\newcommand{\algstart}[2]{\smallskip\noindent{\bf Algorithm} #1. \emph{#2}\begin{enumerate}}
\newcommand{\algend}{\end{enumerate}\vspace{4pt}}
\newcommand{\algitem}{\vspace{2pt}\item}

\begin{document}
\title{Accelerating the CM method}

\author{Andrew V. Sutherland}
\address{Massachusetts Institute of Technology, Cambridge, Massachusetts 02139}
\email{drew@math.mit.edu}


\begin{abstract}
Given a prime $q$ and a negative discriminant $D$, the CM method constructs an elliptic curve $E/\Fq$ by obtaining a root of the Hilbert class polynomial $H_D(X)$ modulo $q$.
We consider an approach based on a decomposition of the ring class field defined by $H_D$, which we adapt to a CRT setting.
This yields two algorithms, each of which obtains a root of $H_D\bmod q$ without necessarily computing any of its coefficients.
Heuristically, our approach uses asymptotically less time and space than the standard CM method for almost all $D$.
Under the GRH, and reasonable assumptions about the size of $\log q$ relative to $|D|$, we achieve a space complexity of $O((m+n)\log q)$ bits, where $mn=h(D)$, which may be as small as $O(|D|^{1/4}\log q)$.
The practical efficiency of the algorithms is demonstrated using $|D| > 10^{16}$ and $q\approx 2^{256}$,
and also $|D| > 10^{15}$ and $q\approx 2^{33220}$.  These examples are both an order of magnitude larger than the best previous results obtained with the CM method.
\end{abstract}

\maketitle

\section{Introduction}
The \emph{CM method} is a widely used technique for constructing elliptic curves over finite fields.
To illustrate, let us construct an elliptic curve $E/\Fq$ with exactly $N$ points.
We shall assume that $q$ is prime, and require $t=q+1-N$ to be nonzero and satisfy $|t|< 2\sqrt{q}$.
We may write $4q=t^2-v^2D$, for some nonzero integer~$v$ and negative discriminant~$D$, and then proceed as follows:

\renewcommand\labelenumi{\theenumi.}
\begin{enumerate}
\item
Compute the Hilbert class polynomial $H_D\in\Z[X]$.
\item
Find a root $x$ of $H_D(X)$ modulo $q$.
\end{enumerate}
The root $x$ is the $j$-invariant of an elliptic curve $E$ with $\#E(\Fq)=N$;
an explicit equation for~$E$ can be obtained via \cite{Rubin:CMTwist}.
The endomorphism ring $\End(E)$ is isomorphic to the imaginary quadratic order $\O$ with discriminant~$D$, and we say that~$E$ has \emph{complex multiplication} (CM) by $\O$.

In principle, the CM method can construct any ordinary elliptic curve $E/\Fq$.
In practice, it is feasible only when $|D|$ is fairly small.  The main difficulty lies in step~1.
The Hilbert class polynomial is notoriously large, as may be seen below:
\vspace{3pt}

\begin{center}
\begin{tabular}{lrrclrr}
$|D|$ & $h(D)$ & size& \hspace{12pt} & $|D|$ & $h(D)$ & size\\\midrule
$10^5+4$    &      152 &  152 KB && $10^{11}+4$ &   145981 &  323 GB \\
$10^6+104$  &      472 & 1.67 MB && $10^{12}+135$&   465872 & 3.53 TB\\
$10^7+47$   &     1512 & 22.3 MB && $10^{13}+15$ &  1463328 & 38.5 TB\\
$10^8+20$   &     5056 &  239 MB && $10^{14}+4$&  4658184 &  384 TB\\
$10^9+15$   &    15216 & 2.73 GB && $10^{15}+15$& 14635920 & 4.45 PB\\
$10^{10}+47$&    48720 & 31.4 GB && $10^{16}+135$& 46275182 & 47.2 PB
\end{tabular}
\end{center}
\vspace{3pt}

The value $h(D)$ is the class number of $D$, which is the degree of $H_D$.
The size listed is an upper bound on the total size of $H_D$ derived from known bounds on its largest coefficient \cite[Lemma~8]{Sutherland:HilbertClassPolynomials}, and is generally accurate to within ten percent.
These discriminants were chosen so that the ratio $h(D)/\sqrt{|D|}$ is within ten percent of its asymptotic average ($0.461559\ldots$), so they represent typical examples.

There are at least three different ways to compute $H_D$: the complex analytic method~\cite{Atkin:ECPP,Enge:FloatingPoint,Gee:GeneratingClassFields}, a $p$-adic approach \cite{Broker:pAdicClassPolynomial,Couveignes:ClassPolynomial}, or by computing $H_D$ modulo many small primes and applying the Chinese remainder theorem (CRT) \cite{Agashe:CRTClassPolynomial,Belding:HilbertClassPolynomial,Chao:CRTCMmethod}.
Under suitable heuristic assumptions all three approaches can achieve quasi-linear running times: $O(|D|\log^c|D|)$ for some constant $c$.
However, the $O(|D|\log^{1+\epsilon}|D|)$ space needed to compute $H_D$ makes it difficult to apply these algorithms when $|D|$ is large.  As noted in \cite{Enge:FloatingPoint}, space is typically the limiting factor.

Here we build on the CRT approach of \cite{Sutherland:HilbertClassPolynomials}, which gives a probabilistic (Las Vegas) algorithm to compute $H_D$, with an expected running time of $O(|D|\log^{5+\epsilon}|D|)$ under the generalized Riemann hypothesis (GRH), and a heuristic running time of $O(|D|\log^{3+\epsilon}|D|)$.   Most critically for the CM method, it directly computes $H_D\bmod q$ without computing $H_D$ over~$\Z$.
This yields a space complexity of $O\bigl(|D|^{1/2+\epsilon}\log q\bigr)$, allowing it to handle much larger values of $|D|$.

As a practical optimization, the CM method may use alternative class polynomials that are smaller than $H_D$ by a large constant factor.
The algorithm in \cite{Sutherland:HilbertClassPolynomials} has recently been adapted to compute such class polynomials \cite{EngeSutherland:CRTClassInvariants}.
To simplify our presentation we focus on the Hilbert class polynomial~$H_D$, but our results apply to all the class polynomials considered in \cite{EngeSutherland:CRTClassInvariants}, a feature we exploit in \S\ref{section:performance}.

When $h(D)$ is composite, a root of $H_D$ can be obtained via a decomposition of the field extension defined by $H_D(X)$, as described in \cite{EngeMorain:Decomposition,HanrotMorain:Solvability}.
The algorithm in~\cite{EngeMorain:Decomposition} computes integer polynomials that describe this decomposition, which can be used in place of $H_D$.
This allows a single root-finding operation to be replaced by two root-finding operations of smaller degree, speeding up step 2 of the CM method. 
We adapt this technique to a CRT setting, where we find it also accelerates step 1, which is the asymptotically dominant step as a function of $|D|$.\footnote{We assume throughout that  fast probabilistic methods are used to find roots of polynomials over finite fields.}

Provided $h(D)$ is sufficiently composite, we may choose a decomposition that significantly reduces the size of the coefficients in the defining polynomials.
In order to do so, we derive an explicit height bound that can be efficiently computed for each of the possible decompositions available.
By choosing the optimal decomposition we gain nearly a $\log|D|$ factor in the running time, on average, based on the heuristic analysis in \S\ref{subsection:heuristics}.
This claim is supported by empirical data, and we give practical examples that achieve more than a tenfold speedup.

We are also able to improve the space complexity of the CM method.
\begin{proposition*}
Assume the GRH, and fix real constants $\delta\ge 0$ and $\epsilon > 0$.
Let~$\O$ be an imaginary quadratic order with discriminant~$D$ and class number $h=mn$, with $m\le O(|D|^{1/2-\delta}$).  Let $q$ be a prime of the form $4q=t^2-v^2D$, and assume $\log q = O(|D|^\delta\log|D|)$.
An elliptic curve $E/\Fq$ with $\End(E)\cong\O$ can be constructed in $O(|D|\log^{6+\epsilon}|D|)$ expected time using $O\bigl((m+n)\log q\bigr)$ space.
\end{proposition*}
This is achieved by interleaving the computation of the defining polynomials modulo many small primes~$p$ with root-finding operations modulo~$q$.
If we additionally require $m = \Omega(|D|^{1/2-\gamma})$, for some $\gamma>\delta$, this yields an $O(|D|^{1/4+\gamma}\log q)$ space bound, improving the $O(|D|^{1/2+\epsilon}\log q)$ result in \cite{Sutherland:HilbertClassPolynomials}.

The organization of the paper is as follows.  We begin with some necessary preparation in \S \ref{section:background}, and
then present two algorithms to obtain a root of the Hilbert (or other) class polynomial modulo a prime $q$ in \S \ref{section:algorithm1} and \S \ref{section:algorithm2}.
The optimization of the height bound is addressed in~\S \ref{section:bound}.
Finally, we present performance data in~\S \ref{section:performance}, including the construction of an elliptic curve over a 256-bit prime field with $|D| > 10^{16}$, and an elliptic curve over a 10000-digit prime field with $|D| > 10^{15}$, both of which are new records for the CM method.

\section{Background}\label{section:background}

\subsection{Hilbert class polynomials}\label{subsection:CMtheory}
We first recall some facts from the theory of complex multiplication, referring to \cite{Cox:ComplexMultiplication, Lang:EllipticFunctions, Serre:ComplexMultiplication} for proofs and further background.
Let $\O$ be an imaginary quadratic order, identified by its discriminant~$D$.  
The $j$-invariant of the lattice $\O$ is an algebraic integer whose minimal polynomial is the \emph{Hilbert class polynomial} $H_D$.
If $\af$ is an invertible $\O$-ideal (including $\af=\O$), then the torus $\C/\af$ corresponds to an elliptic curve $E/\C$ with CM by~$\O$, and every such curve arises in this fashion.
Equivalent ideals yield isomorphic elliptic curves, and this gives a bijection between the ideal class group $\cl(\O)$ and the set
\[
\EllO(\C) = \{j(E/\C)\colon \End(E)\cong\O\},
\]
the $j$-invariants of the elliptic curves defined over $\C$ with CM by $\O$.
We then have
\begin{equation}\label{eq:class1}
H_D(X) = \prod_{j_i\in\EllO(\C)}(X-j_i).
\end{equation}
The splitting field of $H_D$ over $K=\Q(\sqrt{D})$ is the \emph{ring class field} $K_\O$.
It is an abelian extension whose Galois group is isomorphic to $\cl(\O)$, via the Artin map.

This isomorphism can be made explicit via isogenies.
Let $E/\C$ be an elliptic curve with CM by $\O$ and let $\af$ be an invertible $\O$-ideal.
After fixing an isomorphism $\End(E)\cong\O$, there is a uniquely determined separable isogeny whose kernel is the group of points annihilated by every endomorphism in $\af\subset\O\cong\End(E)$.
The image of this isogeny also has CM by~$\O$, and this defines an action of the ideal group of~$\O$ on the set $\EllO(\C)$.
Principal ideals act trivially, and the induced action of the class group is regular.
Thus the set $\EllO(\C)$ is a principal homogeneous space, a \emph{torsor}, for the group $\cl(\O)$.
For a $j$-invariant $j_i$ in $\EllO(\C)$ and an ideal class $[\af]$ in $\cl(\O)$, we write $[\af]j_i$ to denote the image of $j_i$ under the action of $[\af]$.

If $p$ is a prime that splits completely in $K_\O$, equivalently (for $p>3$), a prime that satisfies
the \emph{norm equation}
\begin{equation}\label{eq:norm}
4p = t^2-v^2D,
\end{equation}
for some nonzero integers $t$ and $v$, then $H_D$ splits completely in $\Fp[X]$ and its roots form the set $\EllO(\Fp)$.
Conversely, every ordinary (not supersingular) elliptic curve $E/\Fp$ has CM by some imaginary quadratic order $\O$ in which the Frobenius endomorphism corresponds to an element of norm $p$ and trace $t$.

\subsection{Computing $H_D$ with the CRT method}
The above theory suggests the following algorithm to compute $H_D$ modulo a prime $p$ that splits completely in $K_\O$:
\begin{enumerate}
\item
Find an elliptic curve $E$ with $j$-invariant $j_1\in\EllO(\Fp)$.
\item
Enumerate $\EllO(\Fp)$ as $\{[\af]j_1\colon [\af]\in\cl(\O)\} = \{j_1,\ldots, j_h\}$.
\item
Compute $H_D(X)\bmod p$ as $(X-j_1)(X-j_2)\cdots(X-j_h)$.
\end{enumerate}
Step~1 essentially involves testing curves at random, so this algorithm is feasible only when $p$ is fairly small, but see \cite[\S 3]{Sutherland:HilbertClassPolynomials} for various ways in which this step may be optimized.
Step~2 is addressed in \S \ref{subsection:torsor}.
Step~3 is just polynomial arithmetic, but this is usually the most expensive step \cite[\S 7.1]{Sutherland:HilbertClassPolynomials}.

Under the GRH we can we always work with primes $p$ that are roughly the same size as~$|D|$, no matter how big $q$ is.
We select a sufficiently large set of such small primes that split completely in $K_\O$, and compute $H_D\bmod p$ for each of them.
Provided the product of these primes is larger than $2B$, where~$B$ bounds the coefficients of $H_D$, the polynomial $H_D$ is uniquely determined by the Chinese Remainder Theorem (CRT).
Explicit values for $B$ are given in \cite{Enge:ClassInvariants} and \cite[Lemma~8]{Sutherland:HilbertClassPolynomials}.

However, to compute $H_D\bmod q$ in $O(|D|^{1/2+\epsilon}\log q)$ space, one cannot simply compute $H_D$ over $\Z$ and then reduce it modulo $q$, since writing down $H_D$ would already take too much space.
Instead, one may use the explicit CRT (mod~$q$) \cite{Bernstein:Thesis,Bernstein:ModularExponentiation}, applying it in an online fashion to accumulate results modulo $q$ that are updated after each computation of $H_D\bmod p$, as described in \cite[\S 6]{Sutherland:HilbertClassPolynomials} and summarized in~\S\ref{subsection:crt} below.
Here we also use the explicit CRT, but to obtain a root of $H_D\bmod q$ with a space complexity that may be as small as $O(|D|^{1/4+\epsilon}\log q)$, we must even avoid writing down $H_D \bmod q$.  This requires some new techniques that are described in~\S \ref{subsection:term2}.

\subsection{Explicit CRT}\label{subsection:crt}
Let $p_1,\ldots,p_n$ be primes with product $M$, let $M_i=M/p_i$, and let $a_iM_i\equiv 1 \bmod p_i$.
If $c\in\Z$ satisfies $c\equiv c_i\bmod p_i$, then $c\equiv\sum_i c_ia_iM_i\bmod M$.
If $M > 2|c|$, this congruence uniquely determines~$c$.
This is the usual CRT method.

Now suppose $M > 4|c|$ and let $q$ be a prime (or any integer, in fact).
Then we may apply the \emph{explicit CRT mod $q$} \cite[Thm.\ 3.1]{Bernstein:ModularExponentiation} to compute
\begin{equation}\label{eq:ecrt}
c\equiv \Bigl(\sum_i c_ia_iM_i - rM\Bigr)\bmod q,
\end{equation}
where $r$ is the closest integer to $\sum_i c_ia_i/p_i$; when computing $r$, it suffices to approximate each $c_ia_i/p_i$ to within $1/(4n)$, by \cite[Thm.\ 2.2]{Bernstein:ModularExponentiation}.

When applying the explicit CRT to compute many values $c\bmod q$, say, the coefficients of a polynomial, one first computes the $a_i$ (mod $p_i$) and $M_i$ (mod $q$) using a product tree as described in \cite[\S 6.1]{Sutherland:HilbertClassPolynomials}; this is CRT \emph{precomputation} step, and it does not depend on the coefficients $c$.
Then, as the reduced coefficient values $c_i = c\bmod p_i$ are computed for a particular $p_i$, the sum $\sum c_ia_iM_i\bmod q$ and an approximation to $\sum c_ia_i/p_i$ are updated for each coefficient, after which the $c_i$ may be discarded; this is the CRT \emph{update} step.
Finally, when these sums have been updated for every prime $p_i$, one applies \eqref{eq:ecrt} to obtain $c\bmod q$ for each coefficient, using the sums $\sum c_ia_iM_i$ and the approximations $r\approx\sum c_ia_i/p_i$; this is the CRT \emph{postcomputation} step.
For further details, including explicit algorithms for each step, see \cite[\S 6.2]{Sutherland:HilbertClassPolynomials}

\subsection{Assuming the GRH}\label{subsection:GRH}
The Chebotarev density theorem guarantees that a set of primes $S$ sufficient to compute $H_D$ with the CRT method exists.  We even have effective bounds on their size \cite{Lagarias:Chebotarev}, but these are too large for our purpose.
To obtain better bounds, we assume the GRH (for the Dedekind zeta function of~$K_\O$), which implies that the primes in $S$ are no more than a polylogarithmic factor larger than $|D|$, see \cite[Lemma~3]{Sutherland:HilbertClassPolynomials}.
Having made this assumption, we are then in a position to apply other bounds that depend on some instance of the extended or generalized Riemann hypothesis, all of which we take to be included when we assume the GRH without qualification.
In particular, we make frequent use of the bound $h(D) = O(|D|^{1/2}\log \log |D|)$, proven in~\cite{Littlewood:ClassNumber}, as well as the bound
\begin{equation}\label{eq:sumbound}
\sum_{[\af]\in\cl(\O)}\frac{1}{N(\af)} = O(\log|D|\log\log|D|),
\end{equation}
where $\af$ is the invertible ideal of least norm in $[\af]$, from \cite[Lemma~2]{Belding:HilbertClassPolynomial}.
The sum in (\ref{eq:sumbound}) may also be written as $\sum1/A_i$, where $(A_i,B_i,C_i)$ ranges over the primitive reduced binary quadratic forms $A_ix^2+B_ixy+C_iy^2$ of discriminant $D=B_i^2-4A_iC_i$.

\subsection{Decomposing the class equation}\label{subsection:decomp}
We now describe an explicit method for ``decomposing" a polynomial via a decomposition of the field extension it defines.
This is based on material in \cite{EngeMorain:Decomposition} and \cite{HanrotMorain:Solvability} that we adapt to our purpose here.

Let $K$ be a number field, and let $P\in \Z[X]$ be a monic polynomial, irreducible over $K$, with splitting field $M$.
We have in mind $K=\Q(\sqrt{D})$, $P=H_D$, and $M=K_\O$.
We shall assume the action of $\Gal(M/K)$ is regular, equivalently, that $\idx{M}{K}=\deg P$, which holds in the case of interest.

Given a normal tower of fields $K\subset L\subset M$, we may decompose the extension $M/K$ into extensions
$M/L$ and $L/K$ via polynomials $U\in\Z(Y)[X]$ and $V\in\Z[X]$, where $P(x)=0$ if and only if $U(x,y)=0$ and $V(y)=0$ for some $y\in L$.
The polynomial $V(Y)$ defines the extension $L/K$, and for any root $y$ of $V$, the polynomial $U(X,y)$ defines the extension $M/L$.
In our application $L$ will be identified as the fixed field of a given normal subgroup $G$ of $\Gal(M/K)$.

Let us fix a root $x$ of $P$, and let $\gamma x$ denote the conjugate of $x$ under the action of $\gamma\in\Gal(M/K)$.
Let $\beta_1,\ldots,\beta_n$ be the elements of $G$, and let $\alpha_1G,\ldots,\alpha_mG$ be the cosets of $G$ in $\Gal(M/K)$.
We may factor $P$ in $L[X]$ as $P=P_1\cdots P_m$, where
\begin{equation}\label{eq:Pi}
P_i(X) = \prod_{k=1}^n\bigl(X-\alpha_i\beta_kx\bigr) = \sum_{k=0}^n\theta_{ik}X^k.
\end{equation}

Now let us pick a symmetric function $s$ in $\Z[T_1,\ldots,T_n]$ for which each of the $m$ values $y_i=s(\alpha_i\beta_1x,\ldots,\alpha_i\beta_nx)$ are distinct, equivalently, for which $L=K(y_i)$; Lemma \ref{lemma:symmetric} below shows that this is easily achieved.  We then define the polynomial
\begin{equation}\label{eq:V}
V(Y)=\prod_{i=1}^m\bigl(Y-y_i\bigr).
\end{equation}
The coefficients of $V$ lie in $\Z$, since each may be expressed as a symmetric integer polynomial in the roots of the monic polynomial $P\in\Z[X]$.
For $0\le k\le n$ let
\begin{equation}\label{eq:Wk}
W_k(Y)=\sum_{i=1}^m \theta_{ik}\frac{V(Y)}{(Y-y_i)}.
\end{equation}
As with $V$, we have $W_k\in\Z[Y]$.  Note that $W_k(Y)$ is the unique polynomial of degree less than $m$
for which $W_k(y_i)=\theta_{ik}V'(y_i)$, by the Lagrange interpolation formula.  This definition of the $W_k$ is referred to as the \emph{Hecke representation} in \cite{EngeMorain:Decomposition}.

Finally, let
\begin{equation}\label{eq:U0}
U(X,Y)=\frac{1}{V'(Y)}\sum_{k=0}^nW_k(Y)X^k.
\end{equation}
For each root $y_i$ of $V(Y)$ we then have $U(X,y_i)=P_i(X)$, with $V'(y_i)\ne 0$, since~$V$ has distinct roots.
Each root of $U(X,y_i)$ is a root of $P(X)$, and every root of $P(X)$ may be obtained in this way.
Notice that this construction does not require us to know the coefficients of~$P$, but we must be able enumerate the $G$-orbits of its roots.

As noted in \cite{HanrotMorain:Solvability}, for any given $K$, $P$, and $G$, there are only finitely many linear combinations of the elementary symmetric functions, up to scalar factor, that do not yield a symmetric function $s$ suitable for the construction above.
More precisely, we have the following lemma.

\begin{lemma}\label{lemma:symmetric}
Let $M=K(x)$ be a finite Galois extension of a number field $K$, let $G$ be a normal subgroup of $\Gal(M/K)$
with order~$n$ and fixed field~$L$, and let $m=[L:K]$.
Let $\vec{x}=(x_1,\ldots,x_n)$ denote the $G$-orbit of $x$, and let $e_1, \ldots, e_n$ denote the elementary symmetric functions on $n$ variables.
There are at most $m^{n-1}(m-1)^{n-1}$  linear combinations of the form $s= e_1 + c_2e_2 + \cdots +c_ne_n$, with $c_2,\ldots,c_n\in K$,
for which $L \ne K(s(\vec{x}))$.
\end{lemma}
\begin{proof}
Let $z_k=e_k(\vec{x})$ and $L_k=K(z_1,\ldots, z_k)$.  We have $L_1\subset\cdots\subset L_n = L$, where the last equality follows from the fact that $x$ is a root of the monic polynomial $X^n + \sum_{k=1}^n (-1)^ke_kX^{n-k}$ in $L_n[X]$, since $\idx{L(x)}{L}=n$.
From the proof of the primitive element theorem in \cite[\S 6.10]{Waerden:Algebra1}, we know that $K(z_1,z_2)=K(z_1+c_2z_2)$ for all $c_2\in K$ not of the form $(u_i-u_1)/(v_j-v_1)$, with $j>1$, where  $z_1=u_1,u_2,\ldots,u_r$ are conjugates in $L/K$, and $z_2=v_1,v_2,\ldots,v_s$ are conjugates in $L/K$.
Since $r,s\le m$, there can be at most $m(m-1)$ such $c_2$.  The same argument shows that $K(z_1+c_2z_2+\cdots+c_{k-1}z_{k-1},z_k) = K(z_1+c_2z_2+\cdots+c_kz_k)$ for all but at most $m(m-1)$ values of $c_k\in K$, and the lemma follows.
\end{proof}

In the construction above, if $L=K(s(\vec{x}))$ holds for any $G$-orbit~$\vec{x}$, then it necessarily holds for every $G$-orbit contained in the same $\Gal(M/K)$-orbit, and in this case the $s$-images $y_i$ of these $G$-orbits are distinct, since they are $\Gal(L/K)$-conjugates.
To find such an $s$ explicitly, we pick random integer coefficients $c_2,\ldots,c_n$ uniformly distributed over $[0,2m^2-1]$, and let $s=e_1+c_2e_2+\cdots+c_ke_k$.
By Lemma~\ref{lemma:symmetric}, we then have $L=K(s(\vec{x}))$ with probability at least $1-2^{1-n}$.
In the event that $L\ne K(s(\vec{x}))$, we simply pick a new set of random coefficients and repeat until we succeed, yielding a Las Vegas algorithm.
In the trivial case, $n=1$ and we succeed on the first try with $s=e_1$; for $n>1$ we expect to succeed with at most $1/(1-2^{1-n}) \le 2$ attempts, on average.

\subsection{Orbit enumeration}\label{subsection:torsor}
Recall that $\EllO(\Fp)$ is a torsor for $\cl(\O)$.  Each subgroup $G$ of $\cl(\O)$ partitions $\EllO(\Fp)$ into $G$-orbits that correspond to cosets of~$G$.
Given $j_1\in\EllO(\Fp)$, the set $\EllO(\Fp)$ may be enumerated via \cite[Alg.~1.3]{Sutherland:HilbertClassPolynomials}, but to correctly
identify its $G$-orbits we require some further refinements.

As in \cite[\S 5]{Sutherland:HilbertClassPolynomials}, we represent $\cl(\O)$ using a polycyclic presentation defined by a sequence of ideals $\lf_1,\ldots,\lf_k$ of prime norms whose classes generate $\cl(\O)$.
The \emph{relative order} of $[\lf_i]$ is the least positive integer $r_i$ for which $[\lf_i^{r_i}] \in \langle [\lf_1], \ldots, [\lf_{i-1}] \rangle$.
Every element $[\af]$ of $\cl(\O)$ may be uniquely written in the form
\[
[\af] = [\lf_1^{e_1}]\cdots[\lf_k^{e_k}],
\]
with $0 \leq e_i < r_i$.
Having fixed a polycyclic presentation, we may enumerate $\cl(\O)$ by enumerating the corresponding exponent vectors $(e_1,\ldots,e_k)$.
Moreover, for any subgroup $G$ of $\cl(\O)$, we can easily identify the exponent vectors corresponding to each coset of $G$.
This allows us to distinguish the $G$-orbits of $\EllO(\Fp)$, provided that we can unambiguously compute the actions of $[\lf_1]$,\ldots,$[\lf_k]$.

Let $j_1=j(E_1)$ be an element of $\in\EllO(\Fp)$ and let $\lf$ be an invertible $\O$-ideal of prime norm $\ell \ne p$.
Then $j_2=[\lf]j_1$ is the $j$-invariant of an elliptic curve $E_2$ that is $\ell$-isogenous to $E_1$.
We may obtain $j_2\in\Fp$ as a root of $\phi(X)=\Phi_\ell(j_1,X)$, where $\Phi_N\in\Z[X,Y]$ is the \emph{classical modular polynomial} \cite[\S 69]{Weber:Algebra} that parameterizes cyclic isogenies of degree $N$.
When $[\lf]$ has order 2, there is just one distinct root of $\phi(X)$ that lies in $\EllO(\Fp)$, but in general there will be two: $[\lf]j_1$ and $[\bar\lf]j_1$.
Typically these are the only roots of $\phi(X)$ in $\Fp$, and when they are not, they may be distinguished as the roots that lie on the surface of an $\ell$-volcano, see \cite[\S 4]{Sutherland:HilbertClassPolynomials} or \cite{Ionica:PairingTheVolcano}.
The only ambiguity lies in distinguishing the actions of $[\lf]$ and $[\bar\lf]$, which correspond to the two directions we may walk along the cycle  of $\ell$-isogenies on the surface.

To enumerate $\EllO(\Fp)$ it is not necessary to distinguish these directions, as shown in \cite[Prop.~5]{Sutherland:HilbertClassPolynomials}.
However, it \emph{is} necessary, in general, if we wish to identify the $G$-orbits of $\EllO(\Fp)$ in this enumeration.
To see why, let $\cl(\O)$ be a cyclic group of order 6 generated by $[\af]$, and suppose our polycyclic presentation has $[\lf_1]=[\af^2]$ with $r_1=3$, and $[\lf_2]=[\af]$ with $r_2=2$.
Below are four of the eight possible enumerations of $\EllO(\Fp)$ that might be produced by Algorithm 1.3 of \cite{Sutherland:HilbertClassPolynomials} in this scenario:
\smallskip

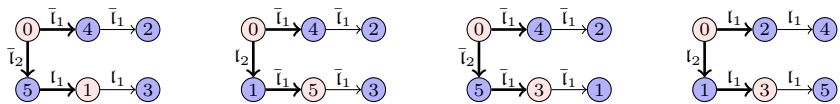
\begin{center}
\begin{tikzpicture}
\scriptsize
  \tikzstyle{vertex}=[circle,draw=black,fill=blue!30,minimum size=9pt,inner sep=0pt]
  \tikzstyle{svertex}=[circle,draw=black,fill=red!10,minimum size=9pt,inner sep=0pt]

  \foreach \name/\x/\y in {            4/0.8/1.6, 2/1.6/1.6,
  												 5/00.0/0.8,            3/1.6/0.8}
    \node[vertex] (G-\name) at (\x,\y) {$\name$};
  \foreach \name/\x/\y in {0/0.0/1.6, 
                                       1/0.8/0.8}
    \node[svertex] (G-\name) at (\x,\y) {$\name$};
  \foreach \from/\to in {0/4,0/5,5/1}
    \draw[line width=1.0pt,->] (G-\from) -- (G-\to);
  \foreach \y in {1.23} \node at (-0.15,\y) {{$\bar\lf_2$}};
  \foreach \x in {0.4,1.2} \node at (\x,1.76) {{$\bar\lf_1$}};
  \foreach \x in {0.4,1.2} \node at (\x,0.95) {{$\lf_1$}};
  \foreach \from/\to in {4/2,1/3}
    \draw[->] (G-\from) -- (G-\to);

 \foreach \name/\x/\y in {            4/3.8/1.6, 2/4.6/1.6,
  												 1/3.0/0.8,            3/4.6/0.8}
    \node[vertex] (G-\name) at (\x,\y) {$\name$};
  \foreach \name/\x/\y in {0/3.0/1.6, 
                                       5/3.8/0.8}
    \node[svertex] (G-\name) at (\x,\y) {$\name$};
  \foreach \from/\to in {0/4,0/1,1/5}
    \draw[line width=1.0pt,->] (G-\from) -- (G-\to);
  \foreach \y in {1.23} \node at (2.85,\y) {{$\lf_2$}};
  \foreach \x in {3.4,4.2} \node at (\x,1.76) {{$\bar\lf_1$}};
  \foreach \x in {3.4,4.2} \node at (\x,0.95) {{$\bar\lf_1$}};
  \foreach \from/\to in {4/2,5/3}
    \draw[->] (G-\from) -- (G-\to);
 
 \foreach \name/\x/\y in {            4/6.8/1.6, 2/7.6/1.6,
  												 5/6.0/0.8,            1/7.6/0.8}
    \node[vertex] (G-\name) at (\x,\y) {$\name$};
  \foreach \name/\x/\y in {0/6.0/1.6, 
                                       3/6.8/0.8}
    \node[svertex] (G-\name) at (\x,\y) {$\name$};
  \foreach \from/\to in {0/4,0/5,5/3}
    \draw[line width=1.0pt,->] (G-\from) -- (G-\to);
  \foreach \y in {1.23} \node at (5.85,\y) {{$\bar\lf_2$}};
  \foreach \x in {6.4,7.2} \node at (\x,1.76) {{$\bar\lf_1$}};
  \foreach \x in {6.4,7.2} \node at (\x,0.95) {{$\bar\lf_1$}};
  \foreach \from/\to in {4/2,3/1}
    \draw[->] (G-\from) -- (G-\to);

 \foreach \name/\x/\y in {            2/9.8/1.6, 4/10.6/1.6,
  												 1/9.0/0.8,            5/10.6/0.8}
    \node[vertex] (G-\name) at (\x,\y) {$\name$};
  \foreach \name/\x/\y in {0/9.0/1.6, 
                                       3/9.8/0.8}
    \node[svertex] (G-\name) at (\x,\y) {$\name$};  \foreach \from/\to in {0/1,0/2,1/3}
    \draw[line width=1.0pt,->] (G-\from) -- (G-\to);
  \foreach \y in {1.23} \node at (8.85,\y) {{$\lf_2$}};
  \foreach \x in {9.4,10.2} \node at (\x,1.76) {{$\lf_1$}};
  \foreach \x in {9.4,10.2} \node at (\x,0.95) {{$\lf_1$}};
  \foreach \from/\to in {2/4,3/5}
    \draw[->] (G-\from) -- (G-\to);

\normalsize
\end{tikzpicture}
\end{center}
\smallskip

Each node corresponds to a $j$-invariant $[\af^e]j_1$ and is labeled by the exponent $e$.
The arrows indicate the action of the ideal that appears in the label, and bold arrows indicate where an arbitrary choice was made.
The two lightly shaded nodes in each configuration correspond to the expected positions of the elements of $\EllO(\Fp)$ corresponding to the subgroup $G=\langle\af^3\rangle$ of order 2.
The two configurations on the right yield a correct enumeration of each $G$-orbit, but the two on the left do not.

\medskip
\noindent
{\bf Remark 1.}
Of course we could have used a polycyclic presentation with $[\lf_1]=[\af]$ and $r_1=6$ in this example, but unless $[\af]$ happens to contain the invertible $\O$-ideal of least prime norm, this is suboptimal, since the cost of computing the action of~$[\lf]$ increases quadratically with the norm of $\lf$.
The optimal polycyclic presentation rarely has one generator, even when $\cl(\O)$ is cyclic.
\medskip

We now consider ways to enumerate $\EllO(\Fp)$ that allow us to identify its $G$-orbits.
First,  the GCD technique of \cite[\S 2.3]{EngeSutherland:CRTClassInvariants} may be used to consistently choose $\lf_i$ or $\bar\lf_i$ each time a new path of $\ell_i$-isogenies is begun (this actually speeds up the enumeration, so it should be done in any case).
We are then left with just $k$ choices, one for each~$\lf_i$.
To consistently orient these choices we may either: (a) use auxiliary relations $[\af_i]=[\lf_1\cdots\lf_i]$, where $\af_i$ has (small) prime norm different from $\ell_1,\ldots,\ell_i$, as described in \cite[\S 4.3]{EngeSutherland:CRTClassInvariants}, or (b) compute the action of Frobenius on the kernel of the two isogenies, as discussed in \cite[\S 4]{BrokerLauterSutherland:CRTModPoly} and the references therein.
Option (a) is fast and easy to implement, but for the best space complexity we should use~(b), since it does not require us to store the entire enumeration.
A minor drawback to option (b) is that it requires $\ell_i\nmid v$, where $v$ is as in the norm equation (\ref{eq:norm}).

However, in many cases neither (a) nor (b) is necessary.
If $G$ is of the form
\begin{equation}\label{eq:goodH}
G=\langle [\lf_1],\ldots,[\lf_{d-1}],[\lf_d^e]\rangle
\end{equation}
for some $d\le k$ and $e|r_d$, then it follows from the proof of \cite[Prop.~5]{Sutherland:HilbertClassPolynomials} that any enumeration of $\EllO(\Fp)$ output by \cite[Alg.~1.3]{Sutherland:HilbertClassPolynomials} also gives a correct enumeration of the $G$-orbits of $\EllO(\Fp)$.
In the context of the CM method, we are free to choose~$G$, and we can always choose one that satisfies (\ref{eq:goodH}).  This may limit our choices for $G$, but in practice this restriction is usually not a burden.

\section{A first algorithm}\label{section:algorithm1}

Our first algorithm is a direct implementation of the theory presented in \S \ref{section:background}, which can significantly accelerate the CM method in the typical case where the class number $h(D)$ is composite.
At this stage we shall not be concerned with improving space complexity; this will be the focus of our second algorithm.
Both algorithms are probabilistic (of Las Vegas type).

As above, $\O$ is an imaginary quadratic order with discriminant~$D$, and we shall assume $h=h(D)>1$.
Let $\PD$ be the set of primes that split completely in $K_\O$, equivalently, primes that satisfy (\ref{eq:norm}).
Given $D$ and a prime $q\in\PD$ we wish to obtain a root of the Hilbert class polynomial $H_D$ over $\Fq$.
Such a root is the $j$-invariant of an elliptic curve $E/\Fq$ with CM by $\O$, as required by the CM method.

The first step is to choose a subgroup $G$ of $\cl(O)$, which determines $n=|G|$ and $m=h/n$.
The quantities $\theta_{ik}$ and $y_i$, $V(Y)$, $W_k(Y)$, and $U(X,Y)$, are then defined as  in \S \ref{subsection:decomp},
with $1\le i \le m$ and $0\le k < n$ (we do not need $k=n$).

The $y_i$ (and therefore $V(Y)$) depend on the choice of a function $s\in\Z[X_1,\ldots,X_n]$ that is a randomly chosen linear combination of the elementary symmetric functions $e_1,\ldots,e_n$, as described at the end of \S \ref{section:bound}.
In the unlikely event that we pick a bad~$s$, we may find that $V$ is a perfect power.
In this case the algorithm simply repeats the computation with a new choice of $s$.

We also require a bound $B$ on the coefficients of $V$ and $W_k$.  The computation of $B$ is addressed in \S \ref{section:bound}.
We now give the algorithm.

\renewcommand\labelenumi{\theenumi.}
\renewcommand\labelenumii{\theenumii.}
\algstart{{\bf 1}}{Given $D$ and $q\in\PD$, find a root $x$ of $H_D$ in $\Fq$ as follows:}
\algitem Select a subgroup $G$ of $\cl(O)$ with $2|G|^2\le q$ and let $n=|G|$.
\algitem Generate random integers $c_2,\ldots,c_n$ uniformly distributed over $[0,2m^2-1]$ and set $s=e_1+c_2e_2+\cdots+c_ne_n$.
\algitem Compute a bound $B$ as described in \S 5.
\algitem As in steps 1-3 of \cite[Alg.~2]{Sutherland:HilbertClassPolynomials}, use $B$ to select $S\subset\PD$, compute a polycyclic presentation $\Gamma$ for $\cl(O)$, and perform CRT precomputation (see \S\ref{subsection:crt}).
\algitem For each $p\in S$:
\begin{enumerate}
\item
Find $j_1\in\EllO(\Fp)$ as in \cite[Alg.~1]{Sutherland:HilbertClassPolynomials}.
\item
Enumerate the $G$-orbits $G_i$ of $\EllO(\Fp)$ using $j_1$ and $\Gamma$.
\item
Compute the $\theta_{ik}$ and the $y_i \bmod p$ (using $s$), as described in \S\ref{subsection:decomp}
\item
Compute $V$ and the $W_k \bmod p$.
\item
Update CRT data for the coefficients of $V$ and the $W_k$.
\end{enumerate}
\algitem Perform CRT postcomputation to obtain $V$ and the $W_k\bmod q$.
\algitem Working in $\Fq$, find a root $y$ of $V$ that is not a root of $V'$.\\
If no such root exists then return to step 2.
\algitem Compute $U_y(X)=U(X,y)\bmod q$ and output a root $x$ of $U_y$ in $\Fq$.
\algend

It follows from \S \ref{subsection:decomp} that the output value $x$ is a root of $H_D$ in $\Fq$, thus the algorithm is correct.
Lemma~\ref{lemma:symmetric} guarantees that it always terminates, and that its expected running time is no more than twice the expected time when the first choice of $s$ works (this factor approaches 1 as $n$ increases).

\medskip
\noindent
{\bf Remark 2.}
For practical implementation one may prefer to fix $s=e_1$ (or $s=-e_1$) and instead change the choice of $G$ if $s=e_1$ does not work.  Using $s=e_1$ simplifies the implementation and can reduce the bound~$B$ significantly.
Empirically, for large values of $|D|$ and $q$, using $s=e_1$ is very likely to work with every choice of $G$.
Alternatively, one may start with $s=e_1$ and then switch to a random $s$ if necessary.
In all our examples, including all the computations in \S \ref{section:performance}, and in the heuristic analysis of \S\ref {subsection:heuristics}, we use $s= \pm e_1$, but our mathematical results (Propositions 1-3) all assume a random  $s$, as specified in Algorithm~1.
\vspace{6pt}

We note three immediate generalizations of Algorithm~1.
First, other class invariants may be treated with suitable modifications to step 2, see \cite{EngeSutherland:CRTClassInvariants}.
Second, it is not necessary for $q$ to be prime; $q$ may be a prime power $q_0^e$ satisfying $4q=t^2-v^2D$ with $t\not\equiv 0\bmod q_0$.
One then computes $V$ and the~$W_k$ mod $q_0$, and performs the root-finding operations in~$\Fq$.
Third, we can compute $V$ and the $W_k$ over $\Z$ using the standard CRT: replace $q$ by the product of the primes in $S$ and lift the results of step 4 to $\Z$.
Steps~7 and~8 may then later be applied to any $q$ that splits completely in the ring class field.

\subsection{Example}\label{subsection:example1}
Let us find a root of $H_D\bmod q$ using Algorithm~1, with $D=-971$ and $q=1029167$.
The class group is cyclic of order 15, and the optimal polycyclic presentation\footnote{The symmetry of the $\ell_i$ and $r_i$ in this small example is entirely coincidental.} has norms $\ell_1=3$ and $\ell_2=5$ and relative orders $r_1=5$ and $r_2=3$.
We choose $G$ to be the subgroup of order $n=5$, which is conveniently of the form~(\ref{eq:goodH}), so we need not distinguish directions when enumerating $\EllO(\Fp)$.  For convenience we set $s=-e_1$, so that $y_i=\theta_{i,n-1}$).

Computing the bound $B$ as in \S \ref{section:bound}, we have $b=\log_2 B \approx 340$ bits, and select a set of primes in $\PD$ whose product exceeds $4B$.
As described in \cite[\S 3]{Sutherland:HilbertClassPolynomials}, we choose primes that optimize the search for $j_1\in\EllO(\Fp)$, obtaining a set of 27 primes:
\[
S = \{263,353,1871,\ldots,38677,43237,62873\}.
\]
We then precompute parameters for the explicit CRT (mod $q$), using \cite[Alg~2.3]{Sutherland:HilbertClassPolynomials}.

For $p=263$ we find $j_1=252$, and enumerate $\EllO(\Fp)$ from $j_1$ as:
\medskip

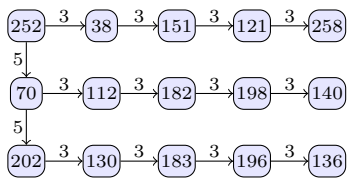
\begin{center}
\scriptsize
\begin{tikzpicture}
  \tikzstyle{vertex}=[shape=rectangle,draw=black,fill=blue!10,minimum size=12pt,inner sep=1pt,rounded corners=4pt]

  \foreach \name/\x/\y in {252/0.0/2.7,  38/1.0/2.7, 151/2.0/2.7, 121/3.0/2.7, 258/4.0/2.7,
                            70/0.0/1.8, 112/1.0/1.8, 182/2.0/1.8, 198/3.0/1.8, 140/4.0/1.8,
						   202/0.0/0.9, 130/1.0/0.9, 183/2.0/0.9, 196/3.0/0.9, 136/4.0/0.9}
    \node[vertex] (G-\name) at (\x,\y) {$\name$};
  \foreach \from/\to in {252/38,38/151,151/121,121/258,70/112,112/182,182/198,198/140,202/130,130/183,183/196,196/136,252/70,70/202}
    \draw[black,->] (G-\from) -- (G-\to);
  \foreach \y in {1.36,2.26} \node at (-0.11,\y) {{$5$}};
  \foreach \x in {0.5,1.5,2.5,3.5} \node at (\x,2.83) {{$3$}};
  \foreach \x in {0.5,1.5,2.5,3.5} \node at (\x,1.93) {{$3$}};
  \foreach \x in {0.5,1.5,2.5,3.5} \node at (\x,1.03) {{$3$}};
\end{tikzpicture}
\end{center}
\medskip
\noindent
The horizontal arrows denote 3-isogenies and the vertical arrows are 5-isogenies.
The three $G$-orbits of $\EllO(\Fp)$ are $\{252,38,151,121,258\}$, $\{70,112,182,198,140\}$, and $\{202,130,183,196,136\}$, corresponding to the rows in the diagram above.
Note that these orbits do not depend on the choice of direction made at the start of each line of isogenies, nor do they depend on the choice of $j_1$.

Continuing with $p=263$, we compute the products
\begin{align*}
P_1(X) &\equiv_p (X-252)(X-38)(X-151)(X-121)(X-258),\\
P_2(X) &\equiv_p (X-70)(X-112)(X-182)(X-198)(X-140),\\
P_3(X) &\equiv_p (X-202)(X-130)(X-183)(X-196)(X-136).
\end{align*}
Each $\theta_{ik}$ is obtained as the coefficient of $X^k$ in the polynomial $P_i$:
\begin{align*}
P_1(X) &\equiv_p X^5 + 232X^4 + 32X^3 + 159X^2 + 208X + 158,\\
P_2(X) &\equiv_p X^5 + 87X^4 + 252X^3 + 139X^2 + 103X + 21,\\
P_3(X) &\equiv_p X^5 + 205X^4 + 86X^3 + 113X^2 + 121X + 116.
\end{align*}
We then set $y_1=\theta_{14}=232$, $y_2=\theta_{24}=87$, and $y_3=\theta_{34}=205$.  Using the values $y_i$ and the $\theta_{ik}$, we compute
\begin{align*}
V(Y) &= (Y-y_1)(Y-y_2)(Y-y_3) \equiv_p Y^3 + 2\thinspace Y^2 + 104\thinspace Y + 59,\\
W_0(Y) &= \sum\theta_{i0}V(Y)/(Y-y_i) \equiv_p 32\thinspace Y^2 + 259\thinspace Y + 152,\\
W_1(Y) &= \sum\theta_{i1}V(Y)/(Y-y_i) \equiv_p 169\thinspace Y^2 + 41\thinspace Y + 153,\\
W_2(Y) &= \sum\theta_{i2}V(Y)/(Y-y_i)  \equiv_p 148\thinspace Y^2 + 117\thinspace Y + 277,\\
W_3(Y) &= \sum\theta_{i3}V(Y)/(Y-y_i)  \equiv_p 107\thinspace Y^2 + 115\thinspace Y + 244,
\end{align*}
We do not need $W_4$ because $y_i=\theta_{i,n-1}$ implies $W_{n-1}(y_i)=y_iV'(y_i)$,
hence the coefficient of $X^{n-1}$ in $U(X,y_i)$ is just $y_i$.
We complete our work for $p=263$ by updating the CRT coefficient data (mod $q$) for the $15$ nontrivial coefficients above, using \cite[Alg~2.4]{Sutherland:HilbertClassPolynomials}.
The same procedure is then applied for each $p\in S$.

Having processed all the primes in $S$, a small postcomputation step \cite[Alg~2.5]{Sutherland:HilbertClassPolynomials} yields $V$ and the $W_k$ modulo $q$:
\begin{align*}
V(Y) &\equiv_q Y^3 + 947907\thinspace Y^2 + 829791\thinspace Y + 760884,\\
W_0(Y) &\equiv_q 975377\thinspace Y^2 + 130975\thinspace Y + 363724,\\
W_1(Y) &\equiv_q 240332\thinspace Y^2 + 135971\thinspace Y + 616131,\\
W_2(Y) &\equiv_q 126738\thinspace Y^2 + 479879\thinspace Y + 908580,\\
W_3(Y) &\equiv_q 340801\thinspace Y^2 + 1000285\thinspace Y + 68659.
\end{align*}
The roots of $V\bmod q$ are $y_1 = 336976$, $y_2 = 898530$, and $y_3=904088$, none of which are roots of $V'\bmod q$.
Using $y=y_1$ we compute 
\begin{align*}
U(X,y) &= X^5 + yX^4 + \frac{1}{V'(y)}\Bigl(W_3(y)X^3 + W_2(y)X^2 + W_1(y)X + W_0(y)\Bigr),\\
     &\equiv_q X^5 + 336976X^4 + 556976X^3 + 849678X^2 + 363260X + 95575,
\end{align*}       
and we then find that $x=590272$ is a root of $U(X,y)$, and hence of $H_D$, modulo~$q$.

This completes the execution of Algorithm~1 on the inputs $D=-971$ and $q=1029167$.
If we now set $k=x/(1728-x) \equiv_q 638472$, then the elliptic curve $E/\Fq$ defined by
\[
Y^2 = X^3 + 3kX + 2k \equiv_q X^3 + 886249X + 247777
\]
has complex multiplication by the quadratic order with discriminant $-971$.
%

\subsection{Complexity}\label{subsection:time1}
The running time of Algorithm 1 has have four principal components.
Let us define $\Tf$, $\Te$, and $\Tb$ (respectively) as the average expected time, over $p\in S$, to: find $j_1\in\EllO(\Fp)$ (step 5a), enumerate the $G$-orbits of $\EllO(\Fp)$ (step 5b), and build the polynomials $V$ and $W_k$ modulo $p$ (steps 5c and~5d).
Additionally, let $\Tr$ be the expected time to find a root $y$ of $V\bmod q$ and to compute and find a root of $U(X,y)\bmod q$ (steps 7 and 8).
As shown in \cite{Sutherland:HilbertClassPolynomials}, the cost of the precomputation in step 4 and the total cost of all CRT computations are negligible, as are steps 1-3.
Thus the total expected running time is
\begin{equation}\label{eq:total1}
O\bigl(|S|(\Tf + \Te + \Tb) + \Tr\bigr).
\end{equation}

When $G$ is trivial, or equal to $\cl(\O)$, Algorithm~1 reduces to the standard CM method, computing $H_D\bmod q$ as in 
\cite[Alg.~2]{Sutherland:HilbertClassPolynomials}.  In this case, under the GRH we have the following bounds:

\setlength{\extrarowheight}{2pt}
\begin{center}
\begin{tabular}{ll}
$|S|$ & $=O(|D|^{1/2}\log\log|D|),$\\
$\Tf$ & $=O(h\log^{5+\epsilon}h),$\\
$\Te$ & $=O(h\log^{5+\epsilon}h),$\\
$\Tb$ & $=O(h\log^{3+\epsilon}h),$\\
$\Tr$ & $=O(h\log h\log^{2+\epsilon} q).$\\
\end{tabular}
\end{center}
\noindent
The first four bounds are proved in Lemmas 7 and 8 of \cite{Sutherland:HilbertClassPolynomials}.

The last bound is based on the standard probabilistic root-finding algorithm of \cite[\S 7]{Berlekamp:PolyFactoringLargeFF}, using fast arithmetic in $\Fp[x]$.
With Kronecker substitution, multiplying two polynomials of degree $d$ in $\Fq[x]$ uses $O(\M(d\log q))$ bit operations, and this
yields an $O(\M(d\log q)\log q)$ bound on the expected time to find a single root of a polynomial in $\Fq[x]$ of degree $d$.
Here $\M(n)$ denotes the time time to multiply two $n$-bit integers \cite[Def.~8.26]{Gathen:ComputerAlgebra}, which we assume to be superlinear, and which satisfies the Sch\"onhage-Strassen bound $\M(n)=O(n\log n\log\log n)$; see \cite{Schonhage:Multiplication}.

When $n=|G|$ properly divides the class number~$h$, the times for $\Tb$ and $\Tr$ may change; these are analyzed below.
The times $\Tf$ and $\Te$ are independent of the choice of $G$, as is the space complexity.
Depending on the bound~$B$, the size of $S$ may also be reduced.
This can lead to a substantial practical improvement, depending on the choice of $G$ and the value of $s$, but we defer this issue to \S \ref{section:bound}.
For the moment we simply note that the bound on $|S|$ above holds for any choice of $G$ and for every~$s$.

We now consider the time $\Tb$.
Computing the $\theta_{ik}$ in step 5c via (\ref{eq:Pi}) involves building $m$ polynomials $P_i$ of degree~$n$ in $\Fp[X]$ as products of their linear factors.
Let $\M(n)$ denote the time to multiply two $n$-bit integers \cite[Def.~8.26]{Gathen:ComputerAlgebra}, which we assume to be superlinear.
With Kronecker substitution, multiplying two polynomials of degree $n$ in $\Fp[x]$ uses $O(\M(n\log p))$ bit operations.
Using a product tree for each $P_i$ yields the bound\footnote{Our bounds count bit operations and hold for all constants $\epsilon > 0$ (and often for $\epsilon=o(1)$).}
\begin{equation}\label{eq:Thetacost}
O(m\M(n\log p)\log n) \subset O(h\log^2 n\log^{1+\epsilon}p)
\end{equation}
on the cost of computing the $\theta_{ik}$.
Here we have used $4p > |D| > h \ge n$ and the $O(n\log n\log\log n)$ bound \cite{Schonhage:Multiplication} for $\M(n)$, which we note applies to the algorithms used in our implementation.
The cost of computing the $y_i$ is also dominated by the time to build $m$ polynomials of degree~$n$ as products of their linear factors: for each $G$-orbit $G_i$ we compute the polynomial $\prod_{j\in G_i}(X-j)\in\Fp[X]$ whose coefficients are the values of the symmetric functions $e_1,\ldots,e_n$ on $G_i$ (up to a sign), from which we compute the linear combination $y_i=s$.

The cost of building $V$ as a product of its linear factors is $O(\M(m\log p)\log m)$, which is dominated by the cost of computing the $n$ polynomials $W_k$ as linear combinations of $V(Y)/(Y-y_i)$ with coefficients $\theta_{ik}$.
Using a recursive algorithm to compute each $W_k$ as in \cite[Alg.~10.9]{Gathen:ComputerAlgebra}, we obtain a total cost of
\begin{equation}\label{eq:Wcost}
O(n\M(m\log p)\log m) \subset O(h\log^2 m\log^{1+\epsilon} p)
\end{equation}
for each iteration of step 5d.
The sum of (\ref{eq:Thetacost}) and (\ref{eq:Wcost}) is $O(h\log^2 h\log^{1+\epsilon} p)$, essentially the same as the cost of building $H_D\bmod p$ from its linear factors.
There is an improvement in the implicit constants, but asymptotically we gain at most a factor of~2 in the time $\Tb$.

We may gain much more in the time $\Tr$.
The total expected time for steps 7 and 8 of Algorithm~1 is bounded by
\begin{equation}\label{eq:tworootcost0}
\Tr = O(\M(m\log q)\log q + h\M(\log q) + \M(n\log q)\log q).
\end{equation}
The three terms in (\ref{eq:tworootcost0}) reflect the time to: (a) find root $y$ of $V\bmod q$ that is not a root of $V'\bmod q$, (b) compute $U_y(X)=U(X,y)\bmod q$, and (c) find a root of $U_y\bmod q$.  In (a), any common roots of $V$ and $V'$ are first removed from $V$ via repeated division by the GCD, which takes negligible time.  We may bound (\ref{eq:tworootcost0}) by
\begin{equation}\label{eq:tworootcost}
\Tr = O(h\log^{1+\epsilon}q + (m+n)\log h\log^{2+\epsilon}q),
\end{equation}
improving $\Tr$ by a factor of $\min\bigl(\log h\log q, mn/(m+n)\bigr)$ compared to the time to find a root of $H_D\bmod q$.
This can reduce the cost of root-finding dramatically, as may be seen in \S \ref{section:performance}.

\section{A second algorithm}\label{section:algorithm2}

Algorithm~1 obtains a root of $H_D\bmod q$ as a root of $U(X,y)\bmod q$,
where $y$ is a root of $V\bmod q$, and $U$ is defined by
\begin{equation}
U(X,Y)=\frac{1}{V'(Y)}\sum_{k=0}^nW_k(Y)X^k.\tag{\ref{eq:U0}}
\end{equation}
To compute $U(X,y)$ we need to evaluate each $W_k$ at $y\in\Fq$.
We observe that it is not necessary to know the coefficients of $W_k$ to do this,
we could instead use
\begin{equation}
W_k(Y)=\sum_{i=1}^m \theta_{ik}\frac{V(Y)}{(Y-y_i)},\tag{\ref{eq:Wk}}
\end{equation}
provided that we know $V$ and the $\theta_{ik}$ (which determine the $y_i$).

Unfortunately we do not know the $\theta_{ik}$ in $\Fq$.
Algorithm~1 computes the $\theta_{ik}$ in $\Fp$, for each prime $p\in S$, but we cannot readily export this knowledge to $\Fq$ because the $\theta_{ik}$ do not correspond to the reductions of rational integers.\footnote{The $\theta_{ik}$ are integers of $G$'s fixed field $L\subset K_\O$.  They do not all lie in $\Q$ unless $G=\cl(\O)$.}
Indeed, the entire reason for using the polynomials $V$ and $W_k$ is that they have coefficients in $\Z$ and are thus defined over any field.

However, there is nothing to stop us from using (\ref{eq:Wk}) to evaluate $W_k$ in $\Fp$.
Given any $z\in\Z$, we can certainly compute $W_k(z)\bmod p$, and if we do this for sufficiently many $p$ we can apply the explicit CRT (mod $q$) to obtain $W_k(z)\bmod q$.

In particular, we can apply this to a lift of $y\in\Fq\cong\Z/q\Z$ to $\Z$.
Explicitly, let $\varphi=\phi\pi$, where $\pi$ is the unique field isomorphism from $\Fq$ to $\Z/q\Z$ and $\phi$ maps each residue class in $\Z/q\Z$ to its unique representative in the interval $[0,q-1]$.
We then have $W_k(\varphi(y))\equiv \varphi(W_k(y))\bmod q$.

Thus it suffices to compute $w_k = W_k(\varphi(y))\bmod q$, and this can be accomplished by computing $w_k\bmod p$ for sufficiently many primes $p$.
Note that while $y$ is a root of $V\bmod q$, when we reduce $\varphi(y)$ modulo $p$, we should not expect to get a root of $V\bmod p$, nor do we need to; we are simply evaluating the integer polynomial $W_k(Y)$ at the integer $\varphi(y)$, modulo many primes $p$.

This leads to our second algorithm, which proceeds in two stages.
The first stage computes $V\bmod q$ and finds a root $y$.
The second stage computes the values $w_k\bmod q$, then computes $U(X,y)\bmod q$ and finds a root $x$.
The second stage requires a bound on $|w_k|$, for which we may use $Bmq^{m-1}$, where $B$ bounds the coefficients of the $W_k$.
For the sake of simplicity we use a single bound for both stages ($mq^{m-1}$ times the bound used in Algorithm 1), but in practice one may compute separate bounds for each stage (in stage 1 it is only necessary to bound the coefficients of $V$).
In order to achieve the best space complexity, certain steps are intentionally repeated, and some may require a more careful implementation, see \S\ref{subsection:space} for details.

As before, the choice of $G\subset\cl(\O)$ in step~1 determines $n=|G|$ and $m=h/n$, and the values $y_i$, $\theta_{ik}$ and $W_k$ are defined for $1\le i\le m$ and $0\le k < n$, where the~$y_i$ depend on the symmetric function $s$ constructed in step 1.

\renewcommand\labelenumi{\theenumi.}
\renewcommand\labelenumii{\theenumii.}
\algstart{{\bf 2}}{Given $D$ and $q\in\PD$, compute a root $x$ of $H_D\bmod q$ as follows:}
\algitem Select $G$, generate a random $s$, and compute $B$, as in Algorithm~1,\\then set $B\leftarrow mq^{m-1}B$.
\algitem Compute a polycyclic presentation $\Gamma$ for $\cl(\O)$.
\algitem Use $B$ to select $S\subset\PD$, and perform CRT precomputation (mod $q$).\footnote{To optimize space this step may need to be performed in an amortized fashion, see \S \ref{subsection:space}.}
\algitem For each $p\in S$:
\begin{enumerate}
\item
Find $j_1\in\EllO(\Fp)$ as in \cite[Alg.~1]{Sutherland:HilbertClassPolynomials}, and cache $j_1(p)=j_1$.
\item
Enumerate the $G$-orbits $G_i$ of $\EllO(\Fp)$ using $j_1$ and $\Gamma$.
\item
Compute the $y_i$ and $V\bmod p$.
\item
Update CRT data for $V\bmod q$.
\end{enumerate}
\algitem Perform CRT postcomputation to obtain $V\bmod q$.
\algitem Find a root $y$ of $V \bmod q$ that is not a root of $V'\bmod q$.\\
If no such root exists then return to step 1.
\algitem For each $p\in S$:
\begin{enumerate}
\item
Let $j_1=j_1(p)$ be the element of $\EllO(\Fp)$ computed in step 4a.
\item
Enumerate the $G$-orbits $G_i$ of $\EllO(\Fp)$ using $j_1$ and $\Gamma$.
\item
Compute the values $\theta_{ik}$, $y_i$, $\phi(y)$ and $V$ mod $p$, and then use them to\\
compute the values $w_k=W_k(\varphi(y))$ mod $p$, via the formula in (\ref{eq:Wk}).\footnote{It is possible to do this without explicitly computing $V\bmod p$, see \S\ref{subsection:example2}.}
\item
Update CRT data for the $w_k\bmod q$.
\end{enumerate}
\algitem Perform CRT postcomputation to obtain the $w_k\bmod q$.
\algitem Compute $U_y(X)=U(X,y)\bmod q$ and output a root $x$ of $U_y\bmod q$.
\algend

Computing the $w_k$ in step 7c via (\ref{eq:Wk}) is faster than computing the coefficients of~$W_k$, by a factor of $\log^2 m$.
For a suitable choice of $G$ (with $n=|G|$ small, say polylogarithmic in $h$), this may improve the time complexity of the entire algorithm, as shown in \S \ref{subsection:time2}.
However, it is often better to choose $G$ to optimize the bound~$B$, as described in \S \ref{section:bound}, which will tend to make $G$ large.

More significantly, computing the scalars $w_k$ rather than the polynomials $W_k$ reduces the space complexity to $O(h\log h + (m+n)\log q)$, which may be much better than the $O(h\log q)$ space complexity of Algorithm~1 when $q$ is large.
This can even be improved to $O((m+n)\log q)$ using a more intricate implementation, but this may increase the time complexity slightly.
The details are given in \S \ref{subsection:space}, where the space complexity of Algorithm~2 is analyzed.


\subsection{Example}\label{subsection:example2}
Let us revisit the example of \S \ref{subsection:example1}, with $D=-971$ and $q=1029167$.
As before, the class group is cyclic of order $h=15$, we let $G$ be the subgroup of order $n=5$, and set $s=-e_1$ so that $y_i=\theta_{i,n-1}$.

The first stage of Algorithm~2 proceeds as in Algorithm~1, except that our height bound is now $\log_2 (mq^{m-1}B)\approx 366$ bits, so we select a slightly larger set $S$, with 29 primes whose product exceeds $4mq^{m-1}B$.
For $p=263$ we again find that the three $G$-orbits of $\EllO(\Fp)$ are $\{252,38,151,121,258\}$, $\{70,112,182,198,140\}$, and $\{202,130,183,196,136\}$.
But instead of computing all the $\theta_{ik}$ as in Algorithm~1, we only need to compute
\begin{align*}
y_1 &\equiv_p -(252 + 38 + 151 + 121 + 258) \equiv_p 232,\\
y_2 &\equiv_p -(70 + 112 + 182 + 198 + 140) \equiv_p 87,\\
y_3 &\equiv_p -(202 + 130 + 183 + 196 + 136) \equiv_p 205,
\end{align*}
which yields
\[
V(Y) = (Y-y_1)(Y-y_2)(Y-y_3) \equiv_p Y^3 + 2\thinspace Y^2 + 104\thinspace Y + 59.
\]
After computing $V\bmod p$ for all $p\in S$ we obtain, via the explicit CRT (mod $q$),
\[
V(Y) \equiv_q Y^3 + 947907\thinspace Y^2 + 829791\thinspace Y + 760884,
\]
and we again find that $y=336976$ is a root of $V$ and not $V'$.
We now lift $y$ from $\Fq\cong\Z/q\Z$ to the integer $\varphi(y)=336976$ and begin the second stage.

For $p=263$ we recompute the three $G$-orbits as above, and this time we compute all of the $\theta_{ik}$, as we did in Algorithm~1.
Setting $y_1=\theta_{14}=49$, $y_2=\theta_{24}=87$, and $y_3=\theta_{34}=205$, we recompute $V\bmod p$ as above.

We now let $z$ be the reduction of $\varphi(y)\bmod p$ and find that $z=73$ (which is not a root of $V\bmod p$, as expected).
We then compute
\begin{align*}
z_1 &= V(z)/(z-y_1) = (z-y_2)(z-y_3)\equiv_p 7,\\
z_2 &= V(z)/(z-y_2) = (z-y_1)(z-y_3)\equiv_p 211,\\
z_3 &= V(z)/(z-y_3) = (z-y_1)(z-y_2)\equiv_p 122.
\end{align*}
The $z_i$ can be computed in two ways: (a) evaluate $V(z)$, invert the $(z-y_i)$, and compute the $z_i$'s as products, or (b) simultaneously compute the $z_i$ as the complements of the $(z-y_i)$ using a product tree, as in \cite[\S 6.1]{Sutherland:HilbertClassPolynomials}.
Both use $O(m)$ field operations, but (b) is faster and works even when $z$ is a root of $V\bmod p$.

We now compute the $w_k$ as linear combinations of the $z_i$ with coefficients $\theta_{ik}$:
\begin{align*}
w_0 &= W_0(z) = \theta_{10} z_1 + \theta_{20} z_2 + \theta_{30} z_3 \equiv_p 227,\\
w_1 &= W_1(z) = \theta_{11} z_1 + \theta_{21} z_2 + \theta_{31} z_3 \equiv_p 79,\\
w_2 &= W_0(z) = \theta_{12} z_1 + \theta_{22} z_2 + \theta_{32} z_3 \equiv_p 44,\\
w_3 &= W_0(z) = \theta_{13} z_1 + \theta_{23} z_2 + \theta_{33} z_3 \equiv_p 242.
\end{align*}
When evaluated at $z$, the polynomials $W_k\bmod p$ computed by Algorithm 1 yield the same values $w_k$ above.
But here we obtained the $w_k$ without computing the~$W_k$, using just $O(h)$ operations to compute them directly from the $y_i$ and the $\theta_{ik}$.
Most importantly, the CRT data used to compute the $w_k\bmod q$ only consumes $O(m\log q)$ space, versus $O(h\log q)$ for the $W_k\bmod q$.

After computing the $w_k\bmod p$ for all $p\in S$, the explicit CRT (mod $q$) yields:
\[
w_0 \equiv_q 180694,\quad w_1\equiv_q 270105,\quad w_2\equiv_q 92440,\quad w_3\equiv_q 110998.
\]
We then evaluate $V'(y)\bmod q$ and use the $w_k$ to compute
\begin{align*}
U(X,y) &= X^5 + yX^4 + \frac{1}{V'(y)}\Bigl(w_3X^3 + w_2X^2 + w_1X + w_0\Bigr)\\
     &\equiv_q X^5 + 336976X^4 + 556976X^3 + 849678X^2 + 363260X + 95575,
\end{align*}
and find that $x=590272$ is a root of $U(X,y)$, and hence of $H_D$, modulo $q$.

\subsection{Time complexity}\label{subsection:time2}
To simplify our analysis we shall initially assume that
\begin{equation}\label{eq:qbound}
m\log q=O(|D|^{1/2}\log|D|).
\end{equation}
Depending on $m=h/|G|$, this may allow $\log q$ to be exponentially larger than $\log |D|$, but here we have in mind the case where $\log q$ is polynomial in $\log |D|$; see \S \ref{subsection:alg2prime} for an approach better suited to large $q$.
Assuming (\ref{eq:qbound}), our bound on $|S|$ is the same as in our analysis of Algorithm~1 in \S \ref{subsection:time1}.
Under the GRH we have
\begin{equation}\label{eq:Scard}
|S| = O(|D|^{1/2}\log\log|D|),
\end{equation}
which bounds the total expected number of iterations in steps 4 and 7.

As with Algorithm 1, the expected running time of Algorithm 2 is bounded by 
\begin{equation}\label{eq:total2}
O\bigl(|S|(\Tf + \Te + \Tb) + \Tr\bigr),
\end{equation}
where $\Tf$, $\Te$, $\Tb$, and $\Tr$ are as defined in \S \ref{subsection:time1}.
The term $\Tf$ is the same for both algorithms, and the term $\Te$ is doubled in Algorithm 2.
The bound $O(h\log^2 h\log^{1+\epsilon}\log p)$ on $\Tb$ in Algorithm 1 becomes
\begin{equation}\label{eq:Tb2}
\Tb = O(m\log^2m\log^{1+\epsilon}p + h\log^2 n\log^{1+\epsilon} p)
\end{equation}
for Algorithm 2, with $p=\max S$.  The first term in (\ref{eq:Tb2}) is the time to build~$V$, while the second is the time to compute the $\theta_{ik}$ and $w_k$.  For suitable $n$, say $n=\log^c h$ for some $c >2$, Algorithm 2 effectively reduces $\Tb$ by a factor of $\log^2 h$.
There is also a minor improvement in $\Tr$; the bound given in (\ref{eq:tworootcost}) becomes
\begin{equation}\label{eq:tworootcost2}
\Tr = O((m+n)\log h\log^{2+\epsilon} q).
\end{equation}

In both the GRH-based and heuristic complexity analyses of \cite[\S 7]{Sutherland:HilbertClassPolynomials}, the bound for $\Te$ dominates the sum $\Tf+\Te+\Tb$.
Thus a better bound on $\Tb$ does not improve the worst-case complexity.
However, the worst-case scenario is atypical, and for almost all $D$ (a set of density 1) this sum is dominated by $\Tb$.
This is heuristically argued in \cite[\S 7]{Sutherland:HilbertClassPolynomials}, and it can be proven using \cite{BaierZhao:QuadraticProgressions} and assuming the GRH (but we will not do so here).

To remove the worst-case impact of $\Te$ (and also $\Tf$), let us consider the performance of Algorithm~2 on a suitably restricted set of discriminants.
Fix positive constants $c_1$, \ldots, $c_7$, and $\epsilon$.
For a positive real parameter $\alpha$, let $\cD(\alpha)$ denote the set of negative discriminants $D$ with the following properties:
\vspace{7pt}
\renewcommand\labelenumi{{\normalfont (\roman{enumi})}}
\begin{enumerate}
\item
The set of integers $p\le c_1|D|\log^{1+\epsilon}|D|$ satisfying $4p=t^2-v^2D$ with $v\ge c_2\log^{1/2+\epsilon/3}|D|$ contains at least $c_3|D|^{1/2}\log^{\epsilon/2}|D|$ primes.
\vspace{5pt}
\item
There are at least two primes $\ell\le c_4\log^{1/2} |D|$ for which $\inkron{D}{\ell}=1$.
\vspace{5pt}
\item
There is a divisor of $h=h(D)$ in the interval $\bigl[c_5\log^\alpha h, \exp (c_6\log^{3/4}h)\bigr]$.
\end{enumerate}
\vspace{7pt}
Conditions (i) and (ii) ensure that for $D\in\cD(\alpha)$, both $\Tf$ and $\Te$ are bounded by $O(|D|^{1/2}\log^{5/2+\epsilon}|D|)$; we refer to \cite[\S 7]{Sutherland:HilbertClassPolynomials} for details.
Provided that $\alpha > 1/2$, for $D\in\cD(\alpha)$ we can choose $G$ so that
\[
c_5\log^{1/2} h\le m\le \exp (c_6\log^{3/4}h) \qquad{\rm and}\qquad n \le h/(c_5\log^{1/2} h),
\]
where $n=|G|$ and $mn=h$.
Applying (\ref{eq:Tb2}), we see that $\Tb$ is then also bounded by $O(|D|^{1/2}\log^{5/2+\epsilon}|D|)$.

To ensure that our assumption in (\ref{eq:qbound}) is satisfied, let us define
\[
\PD(\alpha) = \{q\in\PD : \log q\le c_7\log^{1+\alpha}|D| \}.
\]
This definition is more restrictive than necessary, but for simplicity we impose a uniform bound.
For $q\in\PD(\alpha)$ we then have $\Tr = O(|D|^{1/2+\epsilon})$, which is a negligible component of (\ref{eq:total2}).
This yields the following proposition.

\begin{proposition}\label{prop:time}
Assume the GRH.  For all $\alpha > 1/2$, $D\in\cD(\alpha)$, and $q\in\PD(\alpha)$ there is a choice of $G$ for which the expected running time of Algorithm~$2$ is $O\bigl(|D|\log^{5/2+\epsilon}|D|\bigr)$.
\end{proposition}

\begin{proposition}\label{prop:density}
For all $\alpha<\log 2$, the set $\cD(\alpha)$ has density $1$ in the set of all imaginary quadratic discriminants.
\end{proposition}
\begin{proof}
It suffices to show that each of the properties (i), (ii), and (iii) hold for a set of discriminants $D$ with density 1 (in the set of all imaginary quadratic discriminants).  For (i), this follows from \cite[Thm.~2]{BaierZhao:QuadraticProgressions}.

For (ii), consider just the odd primes $\ell_1,\ldots,\ell_k$ less than $\log\log |D|\le c_4\log^{1/2}|D|$, for sufficiently large $|D|$.  The proportion of congruence classes modulo $L=\prod\ell_i$ corresponding to integers $x$ for which $\inkron{x}{\ell_i}\ne 1$ for all but at most one $\ell_i$ is equal to $$\prod\frac{\ell_i+1}{2\ell_i}+\sum_i \frac{\ell_i-1}{2\ell}\prod_{j\ne i}\frac{\ell_j+1}{2\ell_j} <  (2/3)^k + (k/2)(2/3)^{k-1}=o(1).$$  Thus the number of discriminants in the interval $[-2D,-D]$ that do not satisfy property (ii) is $o(|D|)$.

For (iii), recall that for any $\epsilon>0$ and almost all integers $n$ there are at least $(1-\epsilon)\log\log n$ distinct prime divisors of $n$; see \cite[Thm.\ 431]{Hardy:NumberTheory}.
Therefore almost all discriminants $D$ have at least $k=(1-\epsilon)\log\log |D|-1$ distinct odd prime factors, and for all such discriminants, $h=h(D)$ is divisible by $2^{k-1}$; see \cite[Lem.\ 5.6.8]{Crandall:PrimeNumbers}.
As shown by Siegel, $\log h = (1/2+o(1))\log|D|$, thus for all $\alpha < \log 2$ we can choose $\epsilon >0$ so that $2^{k-1}>c_5\log^\alpha h$ for all sufficiently large $|D|$.
\end{proof}

Propositions~\ref{prop:time} and~\ref{prop:density} together imply that, under the GRH, for almost all discriminants $D<0$ one can find a root of $H_D\bmod q$ in $O(|D|\log^{5/2+\epsilon}|D|)$ expected time, for all $q\in\PD(\alpha)$ with $1/2<\alpha<\log 2$.

We note that the time required to identify and select a suitable $G\subset\cl(\O)$ is negligible by comparison.  When $D$ is fundamental we can obtain a set of generators for the class group $\cl(\O)$ using ideal class representatives of prime norm bounded by $6\log^2|D|$, under the GRH \cite{Bach:ERHbounds}, and for non-fundamental $D$ we also include ideals of prime-power norm for primes dividing the conductor, of which there are $O(\log|D|)$.  Given a generating set $S$ for $\cl(\O)$ of size $O(\log^2|D|)$, we can apply generic algorithms to compute the group structure and an explicit basis for $\cl(\O)$ consisting of elements of prime-power order, in time $O(h^{1/2+\epsilon})$, where $h=|\cl(\O)|=h(D)$, via \cite[Prop.\ 4]{Sutherland:AbelianpGroups}.\footnote{Note that the exponent ${\rm lcm}_{\alpha\in S}|\alpha|$ of $\cl(\O)$ can be computed in time $O(h^{1/2+\epsilon})$; see \cite{Sutherland:thesis}.}
  Once this has been done, it is easy to construct an explicit basis for a subgroup $G$ of any desired order $n$ dividing $h$.

\medskip
\noindent
{\bf Remark 3.}
It is usually better to choose $G$ to optimize the height bound $B$ rather than making the choice required by Proposition~1.
Heuristically, this yields a better improvement in the time complexity, a factor of nearly $\log|D|$ rather than $\log^{1/2}|D|$, see Heuristic Claim~\ref{claim:time} in \S\ref{section:bound}.
\medskip

\subsection{Space complexity}\label{subsection:space}
For convenience we assume the GRH and the restriction (\ref{eq:qbound}) on the size of $q$.
An alternative approach with a weaker restriction on $q$ is given in the next section.
Under these assumptions, $\log p \sim \log|D|$ for all $p\in S$, and $|S|=O(|D|^{1/2}\log\log|D|)$, as in (\ref{eq:Scard}).
A straightforward implementation of Algorithm~2 yields a space complexity of
\begin{equation}\label{eq:naivespace}
O\bigl(|S|\log|D| + |S|\log q + h\log |D| + (m+n)\log q\bigr).
\end{equation}
The first term of (\ref{eq:naivespace}) represents storage for the set $S$, the second term is storage for precomputed data used to apply the explicit CRT (mod $q$), the third term is storage for $\EllO(\Fp)$, and the last term is storage for the $n$ values $w_k$ and the $m$ coefficients of $V$ that are computed via the explicit CRT (mod $q$).
The second term dominates, and we can immediately bound (\ref{eq:naivespace}) by $O(|D|^{1/2}\log\log|D|\log q)$, which is the same as the space complexity of Algorithm~1.

However, the only essential term in (\ref{eq:naivespace}) is the last one, which may be as small as $O(|D|^{1/4}\log q)$.
We now show how to eliminate the first three terms in (\ref{eq:naivespace}).
In practice the most useful term to eliminate (or reduce) is the second one, which requires only a very minor change, but we consider each in turn.

\subsubsection{The first term}\label{subsection:term1}
We wish to avoid storing the entire set $S$ at any one time.
Our strategy is to process $S$ in batches of size $O(|D|^c)$, for some positive $c <1/4$.

In \cite{Sutherland:HilbertClassPolynomials} the set $S$ is chosen by enumerating a larger subset $S_z\subset\PD$ defined by a parameter $z$ that depends on the bound $B$.
The primes in $S_z$ satisfy the norm equation $4p=t^2-v^2D$, where $v$ is $O(\log^{3+\epsilon}|D|)$, and the bound on $t$ depends on~$v$, $h(D)$, and $z$, but is in any case $O(|D|^{1/2+\epsilon})$.
As an alternative to the sieving approach of \cite[Alg.~2.1]{Sutherland:HilbertClassPolynomials}, we enumerate $S_z$ by running through all the integers of the form $(t^2-v^2D)/4$, with $v$ and $t$ suitably bounded, and applying a polynomial-time primality test \cite{Agrawal:PrimesInP} to each.  This takes $O(|D|^{1/2+\epsilon})$ time and negligible space.

Each prime $p$ in $S_z$ is assigned a ``rating" $r(p)$, that reflects the expected cost of finding $j_1\in\EllO(\Fp)$.
The details are not important here, but we may assume that the positive real numbers $r(p)$ are distinctly represented using a precision of $O(\log|D|)$ bits.
Let $r_{\max}$ be the largest (worst) rating among the primes in $S_z$.
For a suitable $r\in[0,r_{\max}]$, we then let $S$ consist of the primes in $S_z$ with ratings bounded by $r$, where $r$ is chosen so that $S$ has the appropriate size.
We can determine such an $r$ space-efficiently using a binary search on the interval $[0,r_{\max}]$, enumerating~$S_z$ at most $\lceil\log_2|S_z|\rceil$ times.
Similarly, we can partition $S$ into batches of size $|D|^c$ by partitioning the interval $[0,r]$ into subintervals whose endpoints are determined using a binary search.
The total cost of computing this partition is $O(|D|^{1-c+\epsilon})$, and this also bounds the cost of enumerating all the batches in $S$.
Thus the time complexity is negligible and we use $o(|D|^{1/4})$ space.

\subsubsection{The second term}\label{subsection:term2}
Let $M$ be the product of the primes $p_i$ in $S$.
As described in \S\ref{subsection:crt}, when applying the explicit CRT (mod $q$), we precompute integers $M_i=M/p_i\bmod q$ and $a_i\equiv (M/p_i)^{-1}\bmod p_i$; this would normally be done in step~3 of Algorithm~2.
The total size of the $M_i$ is $O(|S|\log q)$, which accounts for the second term of (\ref{eq:naivespace}).
Rather than computing the $M_i$ in step 3, we just compute $M\bmod q$ in step 3, and when we need $M_i$ during a CRT update step for $p_i$ (steps 4d and 7d), we invert $p_i\bmod q$ and multiply to obtain $M_i=(M/p_i)\bmod q$.
This reduces the second term of (\ref{eq:naivespace}) from $O(|S|\log q)$ to $O(|S|\log |D|+\log q)$, which is the space used by the $a_i$ and $M\bmod q$.

With the changes described in \S\ref{subsection:term1}, we compute the primes $p_i$ in $S$ in batches of size $|D|^c$, where $c<1/4$.
We now do the same with the $a_i$, computing the $a_i$ for each batch of primes $p_i$ as follows.
If $N$ is the product of the primes in the batch, then it suffices to compute the integer $(M/N)$ modulo $N$, and then simultaneously compute the product of $(M/N)$ and $(N/p_i)$ modulo $p_i$, for all the primes~$p_i$ in the batch, to obtain the $a_i$.
The $a_i$ (and the $p_i$) for a given batch are stored only as long as it takes to process the batch; this means that they are computed twice in Algorithm~2: once in step 4 and then again in step 7.
Using a product tree for each batch, it takes $O(|D|^{1-c+\epsilon})$ time to compute the $a_i$ for all the batches, which is negligible.
With this change, the second term of (\ref{eq:naivespace}) is reduced to $o(|D|^{1/4}) + O(\log q)$.

The only data that is retained once the processing of a given batch of primes has been completed are two values associated to each coefficient $c$ that is to be computed modulo $q$ (the $m$ coefficients of $V\bmod q$ in step 5, and the $n$ values $w_k\bmod q$).  These two values are the partial sums $\sum c_ia_iM_i$ and $\sum c_i a_i/p_i$, where the first is computed modulo $q$ and the latter is stored with a precision of $O(\log q)$ bits, as described in \S\ref{subsection:crt}.  These two values are updated as each prime $p_i$ is processed (across all the batches), and never require more than $O(\log q)$ bits.  There are a total of $m+n$ coefficients, which accounts for the fourth term $O((m+n)\log q)$ listed in (\ref{eq:naivespace}), which also bounds the space required to perform the CRT postcomputation; using the accumulated partial sums, the postcomputation is effectively just a single subtraction to evaluate \eqref{eq:ecrt}.

\subsubsection{The third term}
Steps 4b and 7b of Algorithm~2 both enumerate the $G$-orbits of $j_1\in\EllO(\Fp)$, using the polycyclic presentation $\Gamma$.
Let $j_{ik}$ denote the $k$th element of the $G$-orbit $G_i=(j_{i1},\ldots,j_{in})$, where $i$ ranges from 1 to $m$ and $k$ ranges from 1 to $n$.
There are a total of $h=mn$ values $j_{ik}$, and these account for the $O(h\log|D|)$ third term in \eqref{eq:naivespace}, using the GRH bound $\log p=O(\log|D|)$. 
In order to reduce the space required, we need to process the $j_{ik}$ as they are enumerated, rather than storing them all.
Our basic strategy is to order the enumeration so that we compute the $G$-orbits one by one and discard each $G$-orbit once it has been processed (in fact, we need to enumerate the $G$-orbits twice in step 7 in order to achieve this).
Depending on the presentation $\Gamma$, enumerating the $G$-orbits efficiently may present some complications. These will be addressed below. We first consider how to process the $G$-orbits in a space-efficient manner as they are enumerated.

Let us recall the values we must derive from the $j_{ik}$.
In each iteration of step~4 we compute $m$ values $y_i\bmod p$, each of which is a linear combination of elementary symmetric functions applied to $G_i=(j_{i1},\ldots,j_{in})$.  We then compute the polynomial $V\bmod p$ whose roots are the $y_i$.
In each iteration of step 7 we again compute the $y_i\bmod p$, and also the coefficients $\theta_{ik}$ of the polynomials
\[
P_i(X)=\prod_{k=1}^n=(X-j_{ik})=\sum_{k=0}^n\theta_{ik}X^k
\]
defined in \eqref{eq:Pi}.  Up to a sign, the $\theta_{ik}$ are just the elementary symmetric functions of $j_{i1},\ldots,j_{in}$, so we can derive each $y_i$ from the $\theta_{ik}$.
The $y_i$ are then used to compute $V\bmod p$ (again), and also the values $z_i=V(z)/(z-y_1)$, where $z=\varphi(y)\bmod p$ is the reduction of the integer $\varphi(y)$ corresponding to the root $y$ of $V\bmod q$ computed in step 6.
Finally, we compute the values $w_k=\sum_{i=1}^m\theta_{ik}z_i$, as described in \S\ref{subsection:example2}.

The space required by the $y_i$, the $z_i$, the $w_k$ and the polynomials $V$ and $V$' is just $O((m+n)\log p)$, which is within our desired complexity bound.  We now explain how to process the $G$-orbits in a way that achieves this space complexity.
For each $G_i$ we compute the polynomial $P_i$ with coefficients $\theta_{ik}$, use the $\theta_{ik}$ to compute $y_i$, and then discard $G_i$ and the $\theta_{ik}$.
In step 4 this is all that is required; once all the $G$-orbits have been processed we compute $V(Y)=\prod_{i=1}^m(Y-y_i)$.
In step 7 we proceed as in step 4, and after computing $V$ and the $y_i$ we compute $V'$ and the values $z_i=V(z)/(z-y_i)$.
We then enumerate the $G$-orbits a second time, recompute the $\theta_{ik}$ as above, and then update each of $n$ partial sums $w_k=\theta_{ik}z_i$ by adding the term $\theta_{ik}z_i$.
The space required to process a $G$-orbit is $O(n\log p)$ for the $j_{ik}$ and the $\theta_{ik}$, which are then discarded, plus a total of $O((m+n)\log p)$ space used to store the value $y_i$ and $z_i$, and the partial sums $w_k$ that are retained.
Thus we can process all the $G$-orbits using $O((m+n)\log p)$ space.

We now consider the enumeration of the $G$-orbits.
If the polycyclic presentation~$\Gamma$ for $\cl(\O)$ uses the sequence of ideal classes $[\lf_1],\ldots,[\lf_r]$ and the chosen subgroup $G$ is generated by a prefix of $\Gamma$, say 
\begin{equation}\label{eq:prefixH}
G=\langle [\lf_1],\ldots,[\lf_d]\rangle,
\end{equation}
for some $d < r$, 
then the elements $j_{i1},\ldots, j_{in}$ of the $G$-orbit $G_i$ will appear consecutively in the enumeration of $\EllO(\Fp)$ obtained using $\Gamma$ and no special processing is required.
But the situation in (\ref{eq:prefixH}) is very special, much more so than condition (\ref{eq:goodH}) given in \S \ref{subsection:torsor}.
Indeed, it forces $G$ to be trivial if $r=1$.
We can ensure that $G$ has the form in (\ref{eq:prefixH}) if we construct $\Gamma$ by computing a polycyclic presentation for $G$ and extending it to a polycyclic presentation for~$\cl(\O)$.
Unfortunately, the norms arising in a polycyclic presentation for a proper subgroup of $\cl(\O)$ may be very large: $\Omega(|D|^{1/3})$ in the counterexample of \cite[\S 5.3]{Sutherland:HilbertClassPolynomials}.

To address this, we compute the action of ideals with uncomfortably large norm by representing them as a product of ideals with small norms.
Assuming the GRH, it follows from \cite[Thm.~2.1]{Childs:QuantumIsogenies}  that every element $[\af]$ of $\cl(\O)$ can be expressed in the form $[\af] = [\pf_1 \cdots \pf_t]$, where the $\pf_i$ are ideals of prime norm bounded by $\log^c|D|$, for any $c>2$, and $t = \lceil C\log h/\log\log|D|\rceil = O(\log|D|)$, for some constant $C$ that depends on $c$.
We can find such a representation in $O(h^{1+\epsilon})$ expected time (using negligible space), by simply generating random products of the form $[\pf_1\cdots\pf_t]$ until we find $[\af]$.\footnote{This can be improved to $O(h^{1/2+\epsilon})$ time and $O(h^{1/2-\epsilon})$ space using a birthday-paradox approach with a time/space trade-off, but we don't need to do this.}
After precomputing such a representation in Step 2, we can compute the action of any element of $\cl(\O)$ on any element of $\EllO(\Fp)$ in $O(\log^{6+\epsilon}|D|)$ expected time, allowing us to efficiently handle arbitrarily large norms in $\Gamma$. 

In contrast to the first two terms, removing the third term from (\ref{eq:naivespace}) may increase the overall running time significantly.
The time complexity is increased by a logarithmic factor, but in return, the space complexity may be reduced by an exponential factor.  
We did not make this trade-off in our implementation, as the space complexity of $O(h\log h + (m+n)\log q)$ achieved by simply optimizing the second term is already more than sufficient for the range of $D$ used in the examples of \S \ref{section:performance}.
However for larger computations, increasing the running time by a polylogarithmic factor in order to reduce the space by an exponential factor may be very attractive, especially in a parallel implementation.
\smallskip

Eliminating the first three terms of (\ref{eq:naivespace}) leads to the following proposition, which was summarized in the introduction.
The constraints on $m$ and $q$ ensure that (\ref{eq:qbound}) is satisfied.

\begin{proposition}\label{prop:space}
Assume the GRH, and fix $\delta\ge 0$.
Let $D$ be a discriminant with class number $h=mn$ and let $q\in\PD$, such that $m\le O(|D|^{1/2-\delta}$) and $\log q = O(|D|^\delta\log|D|)$.
With modifications 4.3.1-3, the expected running time of Algorithm~2 on inputs $D$ and $q$ is $O(|D|\log^{6+\epsilon}|D|)$, using $O\bigl((m+n)\log q\bigr)$ space.
\end{proposition}

\subsection{Handling large $\boldsymbol{q}$}\label{subsection:alg2prime}
When $q$ is not constrained by (\ref{eq:qbound}), the bound $mq^{m-1}B$ used by Algorithm~2 to determine the size of $S$ may lead to a significant increase in the running time of Algorithm~2 relative to Algorithm~1, which just uses the bound $B$.
To better handle large $q$ in a space-efficient manner, we make a minor modification to Algorithm~2 that allows us to use the bound $mqB$ instead.
Unless $q$ is extraordinarily large, this will not be significantly different than using the bound~$B$.
We are forced to give up the improved bound on $\Tb$ in (\ref{eq:Tb2}), but the resulting algorithm will still have a running time that is at worst twice that of Algorithm~1, and will typically be only slightly slower.

The key is to avoid exponentiating $\varphi(y)$ in $\Fp$.
Instead we compute all of the powers $y, y^2, y^3,\ldots, y^{m-1}$ in $\Fq$, and then lift these to integers $\varphi(y),\ldots,\varphi(y^{m-1})$ in the interval $[0,q-1]$, which can be reduced modulo $p$.
This is done between steps 6 and 7 of Algorithm~2.
We now modify step 5c to compute the coefficients of $W_k=\sum a_{ik}Y^i \bmod p$ as in Algorithm~1, and then compute values $w_k'\bmod p$ via
\[
w_k' = a_{m-1,k}\varphi(y^{m-1}) + a_{m-2,k}\varphi(y^{m-2}) + \cdots + a_{1,k}\varphi(y) + a_{0,k},
\]
that we use instead of $w_k=W_k(\varphi(y))\bmod p$.
Note that the integers $w_k'$ are not equal to the integers $W_k(\varphi(y))$, but we have $w_k'\equiv W_k(\varphi(y))\bmod q$, which is all that is required, even though $w_k$ and $w_k'$ will typically be distinct modulo $p$.

We may combine this approach with any of the space optimizations considered in the previous section.
When $q$ is very large, the third term of (\ref{eq:naivespace}) is likely to be dominated by the others, so we only optimize the first two terms of (\ref{eq:naivespace}).
Provided $\log q=O(|D|^{1/2})$, the analysis in \cite[\S 7]{Sutherland:HilbertClassPolynomials} then yields an upper bound on the running time of this modified version of Algorithm~2.

\begin{proposition}\label{prop:spacetime}
Assume the GRH.
Let $D$ be a discriminant with class number $h=mn$, and let $q\in\PD$ satisfy $\log q = O(|D|^{1/2})$.
With modifications 4.3.1-2 and 4.4, the expected running time of Algorithm~2 on inputs $D$ and $q$ is $O\bigl(|D|\log^{5+\epsilon}|D|\bigr)$, using $O\bigl((m+n)\log q + h\log h)$ space.
\end{proposition}

\section{Height bounds}\label{section:bound}
In this section we derive an upper bound $B$ on the absolute values of the integer coefficients of $V$ and $W_k$ defined in \S \ref{subsection:decomp}.
More precisely, we compute a height bound~$b$ on the maximum bit-length of any coefficient occurring in $V$ or $W_k$.
This is used by Algorithms~1 and~2 to choose a set of CRT primes $S$ whose product exceeds $4B=2^{b+2}$.

We first derive a general height bound $\bmax$ that depends only on $D$ and can be used with any choice of the subgroup $G\subset\cl(\O)$ and the random symmetric function $s=e_1+c_2e_2+\cdots+c_ne_n$ constructed in Algorithms~1 and 2.
Under the GRH, we have $\bmax = O(|D|^{1/2}\log|D|\log\log|D|)$, which is all that is needed for the proofs of Propositions 1-4.

We then fix $s=e_1$ and derive height bounds that depend not only on $D$, but also on the subgroup $G$.
As may be seen in Table~1, the actual heights can vary significantly with $G$. 
An optimal choice of $G$ may improve the performance of both Algorithms~1 and~2 substantially, provided that we have a height bound that accurately reflects the impact of $G$.
Heuristically, we expect to reduce $b$ by a factor of $\log h/\log\log h$, on average, by computing a customized $b$ for each candidate subgroup $G$ and choosing the best one.
Of course the optimal choice of $G$ depends not only on $b$, but also on how the size of $G$ impacts the complexity 
of building polynomials and finding roots, but when $\log q \ll h$ the height bound is usually the most critical factor.

Throughout this section we work with $j$-invariants, which allows us to use rigorous (and quite accurate) bounds on their size.
Heuristic bounds for other class invariants can be obtained by scaling linearly, as discussed in \S \ref{section:performance}.
We note that all the bounds we derived in this section hold unconditionally; the GRH and the heuristic analysis in \S \ref{subsection:heuristics} are only used to obtain asymptotic growth estimates.

\begin{table}
\begin{tabular}{rr|rr|rr|rr}
$|G|$&bits&$|G|$&bits&$|G|$&bits&$|G|$&bits\\
\midrule
 1& 1983568&  35& 1017514& 182&  639986&   910&  395909\\
 2& 1737305&  39&  955880& 195&  672404&  1001&  444642\\
 3& 1600984&  42&  959237& 210&  649274&  1155&  413905\\
 5& 1464042&  55&  879633& 231&  607751&  1365&  392521\\
 6& 1430692&  65&  841574& 273&  603539&  1430&  402990\\
 7& 1354754&  66&  780769& 286&  540873&  2002&  414360\\
10& 1286551&  70&  877290& 330&  531985&  2145&  422627\\
11& 1235548&  77&  791884& 385&  522887&  2310&  401968\\
13& 1202022&  78&  760840& 390&  540120&  2730&  409766\\
14& 1188816&  91&  756960& 429&  525472&  3003&  436780\\
15& 1195102& 105&  773983& 455&  430383&  4290&  471475\\
21& 1093207& 110&  677448& 462&  487746&  5005&  507403\\
22&  962794& 130&  720919& 546&  492453&  6006&  549648\\
26& 1010539& 143&  697728& 715&  452019& 10010&  756598\\
30& 1006310& 154&  616795& 770&  429293& 15015& 1039684\\
33&  998157& 165&  678832& 858&  437618& 30030& 1983568\\
\bottomrule
\end{tabular}
\vspace{12pt}
\caption{Actual heights for various $G\subset\cl(\O)$ with $D=-221606831$.}
\end{table}

\subsection{Height bound derivations}\label{subsection:heightbound}
As in \S \ref{subsection:decomp}, let $G=\{\beta_1,\ldots, \beta_n\}$ be a subgroup of $\cl(\O)$ with cosets $\alpha_1G,\ldots, \alpha_mG$, so that every element of $\cl(\O)$ is of the form $\alpha_i\beta_k$ with $1\le i\le m$ and $1\le k\le n$.
We may uniquely represent $\alpha_i\beta_k$ by a primitive reduced binary quadratic form $A_{ik}x^2+B_{ik}xy+C_{ik}y^2$ with discriminant~$D$.  If $\tau_{ik}$ denotes the complex number $(-B_{ik}+\sqrt{D})/(2A_{ik})$, we have
\[
H_D(X)=\prod_{i,k}\bigl(X-j(\tau_{ik})\bigr),
\]
where $j(z)$ is the classical modular function.
We assume without loss of generality that $\alpha_1$ and $\beta_1$ are the identity element, and fix $x=j(\tau_{11})$ so that $[\alpha_i\beta_k]x=j(\tau_{ik})$.
We shall use the explicit bound
\begin{equation}\label{eq:jbound}
\bigl|j(\tau_{ik})\bigr| \le \exp\left(\pi\sqrt{|D|}/A_{ik}\right) + 2114.567
\end{equation}
proven in \cite[p.~1094]{Enge:FloatingPoint}, and define $b_{ik}$ to be the logarithm in base 2 (denoted ``$\lg$") of the RHS of (\ref{eq:jbound}), and we note that $b_{ik} > 0$.
We define the \emph{height} $\lht(F)$ of a nonzero polynomial $F(X)=\sum c_jX^j$ in $\C[X]$ by $\lht(F) = \lg\max|c_j|$.
We seek an upper bound $b$ on $\max\{\lht (V),\lht (W_k)\}$ in terms of the $b_{ik}$.

The largest $b_{ik}$ is $b_{11}$, since $\alpha_1\beta_1$ is the identity and we therefore have $A_{11}=1$.
The bound (\ref{eq:jbound}) is nearly tight (one can prove a similar lower bound), and this implies that
$b_{11}\approx \lg(e)\pi\sqrt{|D|}$.
This is about ten times the typical value of $h$, so we shall not be concerned with optimizing terms that are asymptotically smaller than $h$, which we may assume includes both $m$ and $n$, and even $m\lg n$ and $n\lg m$.

As in \S \ref{subsection:decomp}, the $\theta_{ik}$ are defined via
\begin{equation}
P_i(X) = \prod_{k=1}^n\bigl(X-j(\tau_{ik})\bigr) = \sum_{k=1}^n\theta_{ik}X^k,
\end{equation}
where $\prod_{i=1}^m P_i=H_D$.  This yields the bound
\begin{equation}\label{eq:thetabound}
\lht(P_i)\le  n+\sum_{k=1}^n b_{ik}.
\end{equation}
For $V(Y)=\prod_{i=1}^m(Y-y_i)$, with $y_i=s(j(\tau_{i1}),\ldots,j(\tau_{in}))$, we assume that $s$ is of the form $s=\pm(e_1+c_2e_2+\cdots+c_ne_n)$ with $c_k\ge 0$, and define $z_i=|s(2^{b_{i1}},\ldots,2^{b_{in}})|$.  We then have $|y_i|\le z_i$, with $z_i > 1$.  Thus
\begin{equation}\label{eq:Vbound0}
\lht(V) \le m + \sum_{i=1}^m\lg z_i.
\end{equation}
For the nonzero polynomials $W_k(Y)=\sum_{i=1}^m\theta_{ik}\prod_{\ihat\ne i}(Y-y_{\ihat})$, we note that
\[
\sum_{i=1}^m|\theta_{ik}|2^m\prod_{\ihat\ne i}z_{\ihat}
\]
is positive, and bounds every coefficient of $W_k$.  We have
\begin{align}\label{eq:Wbound0}
\lht(W_k) &\le \lg\Bigl(\sum_{i=1}^m|\theta_{ik}|2^m\prod_{\ihat\ne i}z_{\ihat}\Bigr) \le \lg \Bigl( m\max_i|\theta_{ik}|2^m\prod_{\ihat\ne i}z_{\ihat}\Bigr)\\\notag
          &\le m + \lg m + \max_i\Bigl(\lht(P_i) + \sum_{\ihat\ne i}\lg z_{\ihat}\Bigr).
\end{align}
This expression does not depend on $k$, thus it applies to every nonzero $W_k$.

\subsubsection{A general height bound}\label{subsection:crudebound}
For any choice of $s=e_1+c_2e_2+\cdots+c_ne_n$ with $0\le c_k < 2m^2$, we have $z_i\le 2m^2n2^{n-1}2^{b_{i1}+\cdots+b_{in}}$, using $\binom{n}{k}\le 2^{n-1}$, and therefore
\begin{equation}\label{eq:ybound0}
\lg z_i \le n+\lg n+2\lg m + \sum_{k=1}^n b_{ik}.
\end{equation}
From (\ref{eq:Vbound0}) we obtain
\begin{equation}\label{eq:Vbound1}
\lht(V) \le m + mn + m\lg n + 2m\lg m + \sum_{i,k} b_{ik} \le 3h+2h\lg h + \sum_{i,k} b_{ik},
\end{equation}
where we have used $h=mn\ge m\lg n$.
Using the crude bounds $\max_i\sum_k b_{ik}\le \sum_{i,k} b_{ik}$ and $\sum_{\ihat\ne i}\lg z_{\ihat} \le \sum_i\lg z_i$, and applying (\ref{eq:thetabound}) and (\ref{eq:ybound0}) to (\ref{eq:Wbound0}),  we obtain
\begin{align}\label{eq:Wbound1}
\lht(W_k)&\le m + \lg m + n + mn + m\lg n + 2m\lg m + 2\sum_{i,k}b_{ik}\\\notag
          &\le 5h + 2h\lg h + 2\sum_{i,k} b_{ik}.
\end{align}
Since the bound in (\ref{eq:Wbound1}) dominates the bound in (\ref{eq:Vbound1}), we define
\begin{equation}\label{eq:bmax}
\bmax = 5h + 2h\lg h + 2\sum_{i,k} b_{ik},
\end{equation}
where $i$ runs from 1 to $m$ and $k$ runs from 1 to $n$.

Recall that, under the GRH, we have the bounds $h=O(|D|^{1/2}\log\log|D|)$ and $\sum_{i,k}1/A_{ik} = O(\log|D|\log\log|D|)$,
as noted in \S \ref{subsection:GRH}.  These imply
\begin{equation}
\bmax = O(|D|^{1/2}\log|D|\log\log|D|).
\end{equation}

\subsubsection{Optimized height bounds}\label{subsection:optimizedbounds}
We now fix $s= e_1$, which implies $z_i = \sum_{k=1}^n b_{ik}$.
This choice of $s$ minimizes our height bound and simplifies the calculations.

We have
\begin{equation}\label{eq:ybound}
\lg z_i\le\lg\Bigl(\sum_{k=1}^n 2^{b_{ik}}\Bigr)\le \lg n + \max_k b_{ik},
\end{equation}
and from (\ref{eq:Vbound0}) we find that
\begin{equation}\label{eq:Vbound}
\lht(V) \le m + m\lg n + \sum_{i=1}^m\max_k b_{ik}.
\end{equation}
Applying (\ref{eq:thetabound}) and (\ref{eq:ybound}) to (\ref{eq:Wbound0}) yields
\begin{align*}
\lht(W_k) &\le \lg m + m + \max_i\Bigl(n + \sum_k b_{ik}+ \sum_{\ihat\ne i}(\lg n + \max_k b_{\ihat k})\Bigr)\\\notag
					&\le \lg m + m + n + m\lg n + \max_i\Bigl(\sum_k b_{ik}+\sum_{\ihat}\max_k b_{\ihat k} - \max_k b_{ik}\Bigr)\\\notag
					&\le \lg m + m + n + m\lg n + \sum_i\max_k b_{ik} + \max_i\Bigl(\sum_k b_{ik} - \max_k b_{ik}\Bigr),
\end{align*}
where $i$ runs from 1 to $m$ and $k$ runs from 1 to $n$.
This bound dominates the bound in (\ref{eq:Vbound}), so it bounds $\lht(V)$ as well as $\lht(W_k)$.  Thus we define
\begin{equation}\label{eq:bbound}
b = \lg m + m + n + m\lg n + \sum_i\max_k b_{ik} + \max_i\Bigl(\sum_k b_{ik} - \max_k b_{ik}\Bigr)
\end{equation}
as our height bound for $G$, which we typically round up to the nearest integer.

\subsection{Example}
Returning to our example with $D=-971$ and $h(D)=15$, let us compute $b$ for the subgroup $G\subset \cl(\O)$ of order $n=5$.
We can use the same polycyclic presentation $[\lf_1],[\lf_2]$ for $\cl(\O)$ as before, where the ideals $\lf_1$ and $\lf_2$ have norms $\ell_1=3$ and $\ell_2=5$ and $[\lf_1]$ generates $G$, but now we compute directly in the class group, using composition of binary quadratic forms \cite{Buchmann:BinaryQuadraticForms} rather than computing isogenies.
In the notation of \S\ref{subsection:heightbound}, we have $\beta_k=[\lf_1^{k-1}]$ and $\alpha_i=[\lf_2^{i-1}]$.
Enumerating $\cl(\O)$ yields the approximate values
\vspace{3pt}
\begin{center}
\begin{tabular}{lllll}
$b_{11}=141.23$,&$b_{12}=47.08$,&$b_{13}=15.75$,&$b_{14}=15.75$,&$b_{15}=47.08$,\\
$b_{21}=28.25$,&$b_{22}=11.45$,&$b_{23}=20.18$,&$b_{24}=11.96$,&$b_{25}=11.45$,\\
$b_{31}=11.45$,&$b_{32}=11.96$,&$b_{33}=20.18$,&$b_{34}=11.45$,&$b_{35}=28.25$,\\
\end{tabular}
\end{center}
\vspace{3pt}
where each row corresponds to a coset of $G$.
The value of $b_{22}$, for example, is computed using $A_{22}=15$, since $\alpha_2\beta_2=[\lf_2\lf_1]$ is represented by the reduced binary quadratic form $15X^2+13XY+19Y^2$, and we have
\[
b_{22} =\lg\bigl(\exp(\pi\sqrt{971}/15)+2114.567\bigr)\approx 11.45.
\]
To compute $b$ we just need the sum $s_i$ and maximum $t_i$ of $b_{ik}$ over $k$.
These can be computed during the enumeration and stored using $O(m\log|D|)$ space.
We have $s_1=266.89$, $s_2=83.28$, $s_3=83.38$, and $t_1=141.23$, $t_2=28.25$, $t_3=28.25$.
The~$t_i$ sum to 197.73, the maximum of $s_i-t_i$ is 125.65.
Applying (\ref{eq:bbound}) yields
\[
b \approx \lg 3 + 5 + 3 + 3\lg 5 + 197.73 + 125.65 \le 340,
\]
and our height bound is 340 bits.

As noted earlier, we may compute $V$ and the $W_k$ over $\Z$ using Algorithm~1, by computing them modulo some $q\ge2^{b+1}$.
We find that $\max\{\lht(V),\lht(W_k)\}$ is in fact about 324, within five percent of our computed bound $b$.

For comparison, if we instead choose $G$ to be the subgroup of order 3, we obtain $s_1=165.15$, $s_2=55.45$, $s_3=78.71$, $s_4=78.71$, $s_5=55.45$, and $t_1=141.23$, $t_2=28.25$, $t_3=47.08$, $t_4=47.08$, $t_5=28.25$.
The $t_i$ sum to 291.89, the maximum of $s_i-t_i$ is 31.63, and (\ref{eq:bbound}) becomes
\[
b \approx \lg 5 + 3 + 5 + 5\lg 3 + 291.89 + 31.63 \le 342.
\]
If we make $G$ trivial, we have $t_i=s_i=b_{i1}$, the sum of the $t_i$ is 421.51, the maximum of $s_i-t_i$ is zero, and we get $b = 438$, which is nearly the same as the value 434 one obtains from \cite[Lemma~8]{Sutherland:HilbertClassPolynomials}, which uses a more careful analysis than we do here.

\subsection{Computational complexity of optimizing the height bound}
We can compute a basis~$(\gamma_1,\ldots,\gamma_r)$ for the class group $\cl(\O)$ in $O(|D|^{1/4+\epsilon})$ time and space, under the GRH \cite[Prop.~9.7.16]{Buchmann:BinaryQuadraticForms}.  We may assume that each $\gamma_i$ has prime power order $\ell_i^{e_i}$.
We then consider subgroups $G$ of the form
\begin{equation}\label{eq:Gprod}
G=\Bigl\langle\gamma_1^{\ell_1^{d_1}}\Bigr\rangle\times\cdots\times\Bigl\langle\gamma_r^{\ell_r^{d_r}}\Bigr\rangle,
\end{equation}
with $0\le d_i\le e_i$,
which includes subgroups of every possible order $n$ dividing $h=h(\O)$.  There are at most $h$ such subgroups $G$, and enumerating the cosets of $G$ takes $O(h\log^{1+\epsilon}h)$ time, using fast composition of forms \cite{Schonhage:FastForms}.

Thus we can compute a height bound $b$ for every subgroup $G$ of the form in \eqref{eq:Gprod} in $O(h^2\log^{1+\epsilon}h)$ time.
This is $O(|D|\log^{1+\epsilon}|D|)$ under the GRH, which is dominated by our bounds for the running times of Algorithms~1 and 2.
In fact, there are only $O(h^{\epsilon})$ distinct orders that can arise among the candidate subgroups~$G$, and if we restrict our attention to subgroups of the form in (\ref{eq:goodH}), there may be even fewer $G$ to consider.
In practice, the time spent optimizing $b$ is completely negligible (and well worth the effort in any case).

As noted in the example, we only need $O(m\log|D|)$ space to compute the height bound for a given subgroup $G$, which is within the complexity bound of Proposition~3.

\subsection{Heuristic analysis}\label{subsection:heuristics}

Ignoring the minor terms in (\ref{eq:bbound}), the value
\[
b^* = \sum_{i=1}^m\max_k b_{ik} + \max_i\Bigl(\sum_{k=1}^n b_{ik} - \max_k b_{ik}\Bigr)
\]
closely approximates $b$.
When $m=h$ or $n=h$ we have $b^*=\sum_{i,k}b_{ik}$, and in any case $b^*$ is never greater than this sum.
As with $\bmax$, this yields an asymptotic bound for $b^*$ of $O(|D|^{1/2}\log|D|\log\log|D|)$, under the GRH, which then also bounds~$b$.
This is all that can be said in general, since $h$ could be prime.

But $h$ is rarely prime.
This can occur only when $|D|$ is prime, and even then it is unlikely (by Cohen-Lenstra \cite{Cohen:ClassGroupHeuristics}).
Let us consider the typical situation, where we are more or less free to choose the size of $G$, at least up to a constant factor.\footnote{Under the random bisection model, a random integer $N$ in some large interval will have prime-power factors whose logarithms approximate a geometric progression \cite{Bach:thesis}.  One then has divisors of $N$ in most intervals of the form $[M,cM]\subset[1,N]$, for a suitable constant $c$. We heuristically assume that the same applies to $h$.}  Of course $b^*$ depends on the particular choice of $G$, not just its order $n$, but to simplify matters we focus on $n$, and proceed to derive a heuristic estimate for $b^*$ as a function of $n$ and $m=h/n$.

Let us assume that the cosets $G_i$ of $G$ are ordered so that $\min_{k}A_{ik}$ is increasing with $i$ (thus $G_1=G$ contains the identity element $\alpha_1\beta_1$ with $A_{11}=1$).
As a heuristic, let us suppose that, on average, we have
\[
\sum_{i,k}\frac{1}{A_{ik}} \approx \sum_{i,k}\frac{1}{i+(k-1)m} = \sum_{t=1}^h \frac{1}{t} \approx \log h.
\]
That is, we view $\sum 1/A_{ik}$ as an approximation to a harmonic sum in which the terms corresponding to the $i$th coset of $G$ appear at positions $i$, $i+m$, \ldots, $i+(n-1)m$.  This heuristic is based on empirical data collected during the construction of Table~\ref{table:bopt}, which involved analyzing the subgroups of more than 10,000 distinct class groups $\cl(\O)$, with discriminants ranging from $10^5$ to $10^{16}$.
We should emphasize that for any particular choice of $\cl(\O)$ and $G$, the actual situation may deviate quite significantly from this idealized scenario, but if one averages over a large set of class groups and a large sample of their subgroups, one finds, for example, that the average rank of $\min_k A_{ik}$ among all the $A_{ik}$ is approximately $i$,
and we note that the approximation $\sum_{i,k}\frac{1}{A_{ik}}\approx \log h$ is correct to within an $O(\log\log h)$ factor, under the GRH.
In any case, our primary justification for this heuristic is that it yields predictions that work well in practice.

Applying the heuristic yields
\[
\sum_{i=1}^m\max_k \frac{1}{A_{ik}} = \sum_{i=1}^m \frac{1}{i} \approx \log m,
\]
and
\[
\max_i\Bigl(\sum_{k=1}^n \frac{1}{A_{ik}} - \max_k \frac{1}{A_{ik}}\Bigr) = \sum_{k=1}^{n-1}\frac{1}{mk+1} \approx \frac{\log n}{m},
\]
which implies that $b^*$ is within a constant factor of
\[
\left(\log m + \frac{\log n}{m}\right)|D|^{1/2}.
\]
This suggests that if we wish to minimize $b^*$, then we should make $n$ exponentially larger than $m$.
If we let $m\approx\log h$ and $n=h/m\approx h/\log h$, then we expect to have $b^*=O(|D|^{1/2}\log\log|D|)$, improving our worst-case bound by a factor of $\log |D|$, and improving the average case, where $\sum_{i,k}\frac{1}{A_{ik}}\approx \log h$, by a factor of $\log|D|/\log\log|D|$.

Using $m=\log h$ allows us to satisfy the bound (\ref{eq:qbound}) used to analyze Algorithm~2 whenever $\log q=O(|D|^{1/2})$, which is a very mild restriction.  This choice of $m$ precludes the improvement attained by Algorithm~2 under Proposition~\ref{prop:time}, since it makes $n$ too big, but it does lead to the following claim.

\begin{claim}\label{claim:time}
Assuming $\log q = O(|D|^c)$ for some $c<1/2$, the average-case running time of both Algorithms~$1$ and $2$ using $s=e_1$ is $O(|D|^{1/2}\log^{2+\epsilon}|D|)$.
\end{claim}

Empirical data supporting the heuristic analysis above can be found in Table~\ref{table:bopt}.
Each row of the table gives data for 1000 fundamental discriminants of approximately the same size.
We note that for the choice of $G$ that minimizes $b$, the number of cosets $m$ is quite close to $\log h$, on average, as expected.
The two rightmost columns list, respectively, the average and best-case improvement achieved by optimizing the height bound.
The growth rate is consistent with the $\Omega(\log h/\log\log h)$ prediction.
In practice, the actual speedup is substantially better than the height-bound improvement would suggest, for reasons that will be explained in the next section where we analyze practical computations that amply demonstrate the benefit of optimizing the height bound.

\begin{table}
\begin{tabular}{lrrrrrcc}
$N$&$\bar{h}$&$\quad\bar{n}$&$\medspace\quad\bar{m}$&$\quad\bar{b}_h$&$\quad\bar{b}_n$&$\bar{b}_h/\bar{b}_n$&$\max b_h/b_n$\\
\midrule
$10^5$    &      147&      39&  3.7&          7626&       4486& 1.7 & 3.0\\
$10^6$    &      459&      99&  4.6&         28387&      14892& 1.9 & 3.6\\
$10^7$    &     1470&     254&  5.8&        105184&      48667& 2.2 & 4.2\\
$10^8$    &     4632&     671&  6.9&        377174&     157603& 2.4 & 4.9\\
$10^9$    &    14640&    1740&  8.4&       1339636&     509688& 2.6 & 5.5\\
$10^{10}$ &    46434&    4849&  9.6&       4709013&    1644023& 2.9 & 5.6\\
$10^{11}$ &   146598&   14777&  9.9&      16338099&    5374105& 3.0 & 6.1\\
$10^{12}$ &   462979&   41189& 11.2&      56202741&   17182753& 3.3 & 6.6\\
$10^{13}$ &  1460465&  114560& 12.7&     191932881&   54720882& 3.5 & 7.1\\
$10^{14}$ &  4644982&  377059& 12.3&     656497242&  179083436& 3.7 & 7.4\\
$10^{15}$ & 14608895&  964998& 15.1&    2211596515&  560192565& 4.0 & 9.2\\
$10^{16}$ & 46276481& 2695634& 17.2&    7462636834& 1742205583& 4.3 & 9.2\\
\bottomrule
\end{tabular}
\vspace{4pt}
\caption{Height bound optimization}\label{table:bopt}
\vspace{-4pt}
\begin{minipage}{1.0\linewidth}
\small
Each row summarizes data collected for the first 1000 fundamental discriminants $|D|\ge N$.
The value $b_h$ is the unoptimized height bound, corresponding to $|G|=h$, while $b_n$ is the optimized height bound, attained when $|G|=n$.
Bars denote mean values.
\normalsize
\end{minipage}
\end{table}

\section{Computational results}\label{section:performance}
This section presents performance data and computational results.
In order to handle a wider range of discriminants, and to give the most practically relevant examples, we use class invariants derived from various modular functions to which the CRT method has been adapted.
These include, among others, the Weber $\ff$-function, double $\eta$-quotients, and the Atkin functions $A_N$.  We refer to \cite[\S 3]{EngeSutherland:CRTClassInvariants} for definitions of these invariants and a detailed discussion of their implementation using the CRT method.
Here we briefly summarize some key properties of the class invariants we use.

\subsection{Class invariants}
Let $\O=\Z[\tau]$ be an imaginary quadratic order with discriminant $D$, for some $\tau$ in the upper half-plane.
The $j$-invariant $j(\tau)$ is a root of the Hilbert class polynomial $H_D$ and generates the ring class field $K_\O$.
Let $f(z)$ by a modular function of level $N$ related to $j(z)$ by $\Psi_f(f(z),j(z))=0$, where $\Psi_f(F,J)$ is a polynomial with integer coefficients.
The value $f(\tau)$ is an algebraic integer, and when $f(\tau)$ lies in $K_\O$ we call it a \emph{class invariant}.\footnote{We do not require $f(\tau)$ to generate $K_\O$; we can obtain a generator as a root of $\Psi_f(f(\tau),Y)$.}
A given modular function typically yields class invariants only for a restricted subset of discriminants; for example, the primes dividing $N$ must not be inert in $\Q(\sqrt{D})$.

We then define the \emph{class polynomial} $H_D[f]$ by
\[
H_D[f](X)=\prod_{\alpha\in\cl(\O)}\bigl(X-[\alpha]f(\tau)\bigr).
\]
For the functions we consider, $H_D[f]$ has integer coefficients, and the techniques we have developed to find a root of $H_D\bmod q$ apply equally well to $H_D[f]\bmod q$.
Having found a root $f_0$ of $H_D[f]$, we may obtain a root $j_0$ of $H_D$ as a solution to $\psi(Y) = \Psi_f(f_0,Y)=0$.
Since the degree of $\psi$ does not depend on $D$ or $q$, we may bound it by $O(1)$, where the implicit constant depends on $f$.
Thus deriving $j_0$ from~$f_0$ takes just $O(\log^{2+\epsilon} q)$ time.

\subsection{Heuristic height bounds}
The key reason to consider alternative class invariants is that $H_D[f]$ may have much smaller coefficients than $H_D$.
Let us define the \emph{height factor} of $f$ as $c(f)=\deg_F\Psi_f/\deg_J\Psi_f$.  Asymptotically, we have
\[
\lht(H_D[f]) = \frac{\lht(H_D)}{c(f)} + O(1),
\]
where the constant $c(f)$ may be as large as 72.
If $b$ bounds the height of $H_D$, we regard $b/c(f)$ as an approximate bound on the height of $H_D[f]$, but add a small constant (say 256 bits) to account for the $O(1)$ term.
We treat the optimized height bound $b$ computed in \S\ref{section:bound} in the same way.

This heuristic approach may, in rare cases, yield a bound that is too small.
In practice this is easy to detect.  The correct polynomial $V(Y)$ must split completely into linear factors in $\Fq[Y]$, and any sort of random error is extremely likely to yield a polynomial that does not.  Verifying that $V(Y)$ splits into linear factors can easily be incorporated into the root-finding step at no additional cost.\footnote{The first step of the standard root-finding procedure computes the polynomial $\gcd(Y^q-Y,V(Y))$ whose degree is the number of distinct roots of $V$; duplicate roots can be accounted for by taking gcds with derivatives of $V$.}  If $V(Y)$ is found to be incorrect, we may then either retry with a larger height bound, or simply revert to $f=j$ and use the rigorous bound proven in \S\ref{section:bound}.

In most practical applications of the CM method, we seek an elliptic curve $E/\Fq$ with prescribed order $N$, where the prime factorization of $N$ is known (or provisionally assumed).
In this situation we can test whether we have constructed a suitable curve in time $O(\log^{2+\epsilon}q)$, via \cite[Lemma~6]{Sutherland:HilbertClassPolynomials}, which is negligible.
Although it is usually unnecessary, one can also verify the endomorphism ring of the constructed curve, provided that we know the factorization of the integer $v$ in the norm equation $4q=t^2-v^2D$.  Using the algorithm in \cite[Alg.~2]{BissonSutherland:Endomorphism}, this takes time subexponential in $\log|D|$, under heuristic assumptions, which is also negligible.

\subsection{Implementation}
Our tests were performed on a small network of quad-core AMD Phenom II 945 CPUs, each clocked at 3.0~GHz.
The computation of class polynomials (or decompositions thereof) was distributed across up to 48 cores, depending on the size of the test, with essentially linear speedup, while all root-finding operations were performed on a single core.
For consistency we report total CPU times, summed over all threads.

The software was implemented using the \texttt{gmp}~\cite{GMP} and \texttt{zn\_poly}~\cite{Harvey:zn_poly} libraries, with the \texttt{gcc}~compiler~\cite{GNU}.
Polynomial arithmetic modulo the small primes $p\in S$ was handled via \texttt{zn\_poly}, while polynomial arithmetic modulo large primes $q$ used the cache-friendly truncated FFT approach described in \cite{Harvey:CacheFriendly}, layered on top of the \texttt{gmp} library.
In order to simplify the implementation, when selecting the subgroup $G$ to optimize the height bound, only subgroups of the form~(\ref{eq:goodH}) in \S \ref{subsection:torsor} were considered.
Additionally, of the various space optimizations described in \S \ref{subsection:space} that may be applied to Algorithm~2, only the changes necessary to achieve a space complexity of $O(h\log h + (m+n)\log q)$ were used (see \S \ref{subsection:term2}).
A more complete implementation would improve some of the results presented here.

As noted in Remark~2, in our implementation we fixed $s=e_1$.  This choice of~$s$ worked in every large ($|D| > 10^6$, $\log q > 160$) example that we tested, which included more than a million different combinations of $D$ and $G$.  We conjecture that $s=e_1$ always works when using $j$-invariants, but note that it can fail for other class invariants in rare cases (the handful of exceptions we found all involved very small discriminants, and in each such case switching to $s=e_2$ worked).

\begin{table}
\begin{tabular}{@{}l@{}rrr@{}}
&Example 1& Example 2& Example 3\\
\midrule
Discriminant $|D|$ & $13569850003$ & $\quad\medspace 11039933587$ & $\quad\medspace 12901800539$\\
Field size $\lceil\lg q\rceil$ & 177 & 231 & 172\\
\vspace{1pt}
Class number $h$& 20203 & 11280 & 54076\\
Presentation $\ell_1^{r_1},\ldots,\ell_k^{r_k}$&$7^{20203}$&$17^{1128},19^{10}$&$3^{27038},5^2$\\
Modular function $f$                         & $A_{71}$& $A_{47}$& $A_{71}$\\
Height factor $c(f)$                 &    36 &    24 &    36\\
\midrule
{\bf Standard}&&&\\
Subgroup size $|G|=h$&20203&11280&54706\\
Height bound $b_h$                         & 63127 & 56631 & 151939\\
Number of primes $|S|$&1993&1783&4477\\
$\Tf$ \hspace{5pt}(ms)&48&110&42\\
$\Te$ (ms)&33&48&23\\
$\Tb$ \hspace{1pt}(ms)&15&7&63\\
\cmidrule(r){2-4}
$\Tp$ (s) &197&295&597\\
$\Tr$ \hspace{1pt}(s) & 56&54&171\\
$\Tt$ \hspace{5pt}(s) &\textbf{253}&\textbf{347}&\textbf{768}\\
\midrule
{\bf Accelerated}&&&\\
Subgroup size $|G|=n$&227&1128&2458\\
Height bound $b_n$                         & 35115 & 30957 & 50180\\
Number of primes $|S|$&1115&994&1519\\
$\Tf$ \hspace{5pt}(ms)&44&105&28\\
$\Te$ (ms)&33&47&23\\
$\Tb$ \hspace{1pt}(ms)&6&3&22\\
\cmidrule(r){2-4}
$\Tp$ (s) &95&155&118\\
$\Tr$ \hspace{1pt}(s) & $0$&4&5\\
$\Tt$ \hspace{5pt}(s) &\textbf{95}&\textbf{159}&\textbf{123}\\
\bottomrule
\end{tabular}
\\
\vspace{5pt}
\caption{Example CM constructions}\label{table:examples}
\end{table}

\begin{table}
\begin{tabular}{@{}rrrrcrrcrr@{}}
&&&&&\multicolumn{2}{c}{{\bf Standard}}&&\multicolumn{2}{c}{{\bf Accelerated}}\\
\cmidrule(r){6-7}\cmidrule(r){9-10}
$|D|$&$h$&$n$&$b_h/b_n$&&$\Tp$&$\Tr$&&$\Tp$&$\Tr$\\
\midrule
     6961631 &    5000 &   250 &  3.63 &&    1.0 &    25 &&    0.2 &  0.7\\
    23512271 &   10000 &   250 &  3.65 &&    3.9 &    58 &&    0.8 &  0.7\\
    98016239 &   20000 &   625 &  4.10 &&     21 &   126 &&    3.5 &  2.1\\
   357116231 &   40000 &   625 &  4.82 &&     90 &   282 &&     11 &  2.2\\
  2093236031 &  100000 &  2500 &  4.93 &&    750 &   812 &&     88 &   11\\
  8364609959 &  200000 &  4000 &  6.34 &&   3590 &  1805 &&    301 &   16\\
 17131564271 &  300000 &  6250 &  6.47 &&   9070 &  2890 &&    708 &   19\\
 30541342079 &  400000 & 12500 &  6.31 &&  16900 &  3910 &&   1380 &   71\\
 42905564831 &  500000 & 15625 &  7.11 &&  28300 &  4410 &&   2300 &   85\\
170868609071 & 1000000 & 25000 &  7.06 && 123000 &  9260 &&   8840 &  159\\
\bottomrule
\end{tabular}
\\
\vspace{5pt}
\caption{CM constructions using the Weber $\ff$-function and $q\approx 2^{256}$}\label{table:Weber}
\end{table}

\subsection{Accelerated CM computations with Algorithm~1}\label{subsection:alg1perf}

We applied Algorithm~1 to several examples that have previously appeared in the literature.
The examples in Table~\ref{table:examples} are taken from \cite[Table~2]{Sutherland:HilbertClassPolynomials} where they appear as representatives of a large set of computations to construct elliptic curves suitable for pairing-based cryptography.
These examples are also used in \cite[Table~1]{EngeSutherland:CRTClassInvariants} with the class invariants we use here.\footnote{The timings listed here for the standard computations are slightly better (about 5\%) than those in \cite{EngeSutherland:CRTClassInvariants} due to a more recent version of \texttt{gmp}.}
The first five discriminants in Table~\ref{table:Weber} originally appeared in \cite[Table~1]{Enge:FloatingPoint}, and can also be found in \cite[Table~4]{Sutherland:HilbertClassPolynomials} and \cite[Table~2]{EngeSutherland:CRTClassInvariants}.
The remaining discriminants are from \cite[Table~4]{Sutherland:HilbertClassPolynomials}.

The time $\Tp$ listed in Tables~\ref{table:examples}-\ref{table:D15perf} is the total time spent computing the polynomial $V$ and the polynomials $W_k$, in the case of Algorithm~1 (steps 1-6), and the total time spent computing the polynomial $V$ and the values $w_k$ in the case of Algorithm~2 (steps 1-5 and 7-8), including all precomputation.
The time $\Tr$ is the time spent on root-finding operations (steps 7-8 in Algorithm~1 and steps 6 and 9 in Algorithm~2).
For the smaller examples, these are averages over 10 runs; with the large examples there is very little variance in the root-finding times.

The ``{\bf Standard}" computations listed in Tables~\ref{table:examples} and~\ref{table:Weber} correspond directly to the computations in \cite{EngeSutherland:CRTClassInvariants}, and are equivalent to running Algorithm~1 with $G=\cl(\O)$.
The ``{\bf Accelerated}" computations used Algorithm~1 with $G$ chosen to minimize $\Tt$, based on heuristic formulas for $\Tp$ and $\Tr$ extrapolated from empirical data.
In most cases this minimizes the corresponding height bound, but not always;
in the $h=100000$ example of Table~\ref{table:Weber}, using $n=5000$ rather than $n=2500$ improves the ratio $b_h/b_n$ from 4.93 to 5.84 and reduces $\Tp$ by 10 seconds, but it increases $\Tr$ by 21 seconds, so this subgroup was not chosen.

The $\Tp$ times listed in Table \ref{table:examples} are in each case slightly greater than the quantity $|S|(\Tf+\Te+\Tb)$, due to time spent updating the CRT data in step 3e of Algorithm~1, which is included in $\Tp$.
This difference is only a few percent for the values of~$q$ used in these examples, but becomes more significant when $q$ is very large (see \S\ref{subsection:alg2perf}).

The third example in Table~\ref{table:examples} illustrates four ways in which Algorithm~1 can reduce the time required to apply the CM method using the CRT approach:
\renewcommand\labelenumi{\theenumi.}
\begin{enumerate}
\item
The height bound $b_n=50180$ is about 3 times smaller than $b_h=151359$, which reduces $|S|$ similarly, from 4477 to 1519.
\item
The average time $\Tf$ spent finding an element of $\EllO(\Fp)$ is reduced from 42 to 28 milliseconds,
because the primes that remain in $S$ are those for which it is easier to find curves in $\EllO(\Fp)$.
\item
The average time $\Tb$ spent building polynomials from their roots (or computing linear combinations) is reduced from 63 to 22 milliseconds, because the degrees of the polynomials involved are $m=22$ and $n=2458$ rather than $h=54076$.
\item
Working with polynomials of lower degree reduces the time $\Tr$ spent finding roots dramatically: from 171 seconds to 5 seconds.
\end{enumerate}

As may be seen in Tables~\ref{table:examples} and~\ref{table:Weber}, the speedup achieved by Algorithm~1 is typically better than the height bound ratio $b_h/b_n$, for the reasons noted above.
In the last example of Table~\ref{table:Weber}, with discriminant $D=-170868609071$ and class number $h=1000000$, computing an optimized height bound~$b_n$ with $n=25000$ improves the height bound by a factor of about~7, but $\Tp$ is reduced by nearly a factor of 14 and $\Tr$ is reduced even more.

The discriminant $D=-170868609071$ also appears in \cite[Table~4]{Sutherland:HilbertClassPolynomials}, which lists a time equivalent to 150 CPU days on our current test platform to compute the Hilbert class polynomial $H_D$ modulo a 256-bit prime $q$.
Here we instead use the Weber $\ff$-function, with a height factor of 72, and are able to further improve the height bound by a further factor of 7 using a decomposition of the class polynomial $H_D[\ff]$.
We eventually obtain a root of the original polynomial $H_D\bmod q$, and it takes only 2.5 CPU hours to do so, an overall speedup by nearly a factor of 1500.

\subsection{Optimizing space with Algorithm 2}\label{subsection:alg2perf}

As noted in Remark 2, choosing~$G$ to optimize the height bound may negate any performance advantage Algorithm~2 might have over Algorithm~1.
Indeed, Algorithm~1 is usually faster, due to the larger height bound required by Algorithm~2, and the fact that Algorithm~2 repeats the enumeration step in its second stage.
However when $q$ is large, Algorithm~2 may use much less space than Algorithm~1, which can actually lead to a better running time.
In this scenario we use the modified form of Algorithm~2 described in \S\ref{subsection:alg2prime}, which makes the height bound increase negligible, and to optimize space we choose~$G$ so that $m=h/n$ is approximately equal to but no larger than $n=|G|$.

Table \ref{table:alg1v2} compares the time and space required by Algorithms 1 and 2 for a fixed discriminant $D=-300000504611$ and increasingly large primes $q\approx 2^k$.
The class number is $h=2^{18}$, and in each case we choose $G$ so that $m=n=2^9$.
In addition to the times $\Tf$, $\Te$, and $\Tb$ listed in Table \ref{table:examples}, we also list the average time $\Tc$ spent updating CRT data for each of the primes $p\in S$.
This time is negligible when~$q$ is of moderate to cryptographic size (say, up to 1024 bits), but when~$q$ is very large, as may occur in elliptic curve primality proving \cite{Atkin:ECPP,Morain:ECPP}, the time Algorithm~1 spends updating its CRT data becomes quite significant.\footnote{As discussed in \cite[\S 6.3]{Sutherland:HilbertClassPolynomials}, we should eventually transition from the explicit CRT to a standard CRT approach as $q$ grows, but here $\lg q$ is still much smaller than the height bound $b$.}

One can see the two disadvantages of Algorithm~2 in Table~\ref{table:alg1v2}; it requires a slightly larger $S$, and $\Te$ is doubled.
However Algorithm~2 needs much less space for its CRT data, and spends negligible time updating it.
In our implementation both Algorithms~1 and~2 use $O(h\log|D|)$ space for the computations performed modulo each prime $p\in S$, about 10~MB in this example, but Algorithm~1 requires $O(h\log q)$ space for its CRT data, regardless of the choice of $G$, whereas Algorithm~2 only requires $O((m+n)\log q)$ space.
As shown in the last two rows of Table~\ref{table:alg1v2}, for $q\approx 2^{32768}$ Algorithm~1 uses more than 1 GB of CRT data, compared to about 4~MB for Algorithm 2, which leads to a significant time advantage for Algorithm~2.
Note that the memory required by Algorithm~2 to store its CRT data is actually half the size listed in Table~\ref{table:alg1v2}, since with $m=n$ the CRT data is evenly split across the 2 stages and the CRT data for the first stage can be discarded before the second stage begins.

Due to its superior space complexity, Algorithm~2 is able to effectively handle a broader range of $|D|$ and $q$ than Algorithm~1.
The next section gives an example of a computation with $|D|\approx 10^{15}$ and $q\approx 10^{10000}$ that is easily handled by Algorithm~2 but would be impractical to compute on our test platform using Algorithm~1, or any algorithm that requires space proportional to the size of $H_D\bmod q$.

\begin{table}
\begin{tabular}{@{}rlrrrrrrrrr@{}}
$\lg k$&alg&$b$&$|S|$&$\Tf$&$\Te$&$\Tb$&$\Tc$&$\Tp$&$\Tr$&CRT data\\
       &   & (bits)  &  &(ms)&(ms)&(ms)&(ms)&(s)&(s)&(MB)\\
\midrule
 7&1& 378315& 10013& 165& 107& 166& 25&  4660&    1& 8.2\\\vspace{4pt}
  &2& 378452& 10016& 166& 213& 168&  0&  5510&    1& 0.03\\
 8&1& 378315& 10013& 165& 107& 166& 26&  4660&    3& 12.6\\\vspace{4pt}
  &2& 378350& 10020& 166& 213& 169&  0&  5510&    3& 0.05\\
 9&1& 378315& 10013& 165& 107& 165& 28&  4670&   12& 21.0\\\vspace{4pt}
  &2& 378836& 10026& 166& 213& 169&  0&  5520&   12& 0.08\\
10&1& 378315& 10013& 165& 107& 166& 33&  4270&   37& 37.7\\\vspace{4pt}
  &2& 379348& 10039& 166& 213& 169&  0&  5530&   37& 0.15\\
11&1& 378315& 10013& 165& 107& 166& 43&  4820&  142& 71.3\\\vspace{4pt}
  &2& 380372& 10066& 166& 213& 169&  0&  5540&  142& 0.28\\
12&1& 378315& 10013& 165& 107& 166& 73&  5120&  697& 138\\\vspace{4pt}
  &2& 382420& 10119& 166& 213& 169&  0&  5590&  697& 0.54\\
13&1& 378315& 10013& 165& 107& 166&129&  5690& 3420& 273\\\vspace{4pt}
  &2& 386516& 10225& 167& 213& 169&  0&  5630& 3420& 1.06\\
14&1& 378315& 10013& 165& 107& 166&225&  6700&16510& 541\\\vspace{4pt}
  &2& 394708& 10437& 168& 214& 170&  1&  5810&16510& 2.11\\
15&1& 378315& 10013& 165& 107& 166&461&  9100&81100&1078\\\vspace{4pt}
  &2& 411902& 10859& 170& 214& 171&  2&  6060&81100& 4.21\\
\bottomrule
\end{tabular}
\\
\vspace{5pt}
\caption{Algorithms 1 and 2 with $n=|G|=512$ and $q\approx 2^k$}\label{table:alg1v2}
\vspace{-2pt}
\begin{minipage}{1.0\linewidth}
\begin{center}
\small
$D=-300000504611$ with $h(D)=262144$ using $A_{71}$.
\end{center}
\end{minipage}
\end{table}

\subsection{Some large examples}\label{subsection:largetests}

We also tested Algorithms~1 and 2 with some larger discriminants, beginning with $D=-1000000013079299$, which has class number $h(D)=10034174$.
This discriminant appears in \cite{EngeSutherland:CRTClassInvariants}, where it was used to construct an elliptic curve over a 256-bit prime field via a class invariant derived from the Atkin function $A_{71}$.
As noted in \cite{EngeSutherland:CRTClassInvariants}, this set of parameters was chosen so that the level $N=71$ is ramified in $\Q(\sqrt{D})$, which allows us to work with the square root of the class polynomial $H_D[A_{71}]$, reducing both the degree and the height bound by a factor of two.
The decomposition techniques described here can be applied directly to the polynomial $\sqrt{H_D[A_{71}]}$, allowing both Algorithms~1 and~2 to take advantage of this situation.

Table~\ref{table:D15perf} gives timings for five computations that constructed elliptic curves modulo a 256-bit prime $q$ by obtaining a root of the polynomial $\sqrt{H_D[A_{71}]}\bmod q$.
The first row corresponds to the original computation in \cite{EngeSutherland:CRTClassInvariants}.
The next two rows give timings for Algorithms~1 and~2 when the subgroup $G$ is chosen to optimize the running time of Algorithm~1, with $n=|G|=44399$.
This reduced the total CPU time by nearly a factor of 5, allowing the entire computation to be completed in less than a day of elapsed time on 48 cores.
The portion of CPU time spent on root-finding was cut dramatically, from more than a day to under five minutes.
This improvement is particularly helpful in a distributed implementation, as root-finding is not as easy to parallelize as the other steps and is most conveniently performed on a single CPU.

The last two rows of Table~\ref{table:D15perf} give timings for Algorithms~1 and~2 when $G$ is chosen to optimize the space used by Algorithm~2.
This increases the running time by about 15\%, but requires less than 2 MB of CRT data, compared to about 250 MB for the original computation (and Algorithm~1).
This reduced the total memory usage from around 500 MB to about 100 MB.

As noted in \S\ref{subsection:alg2perf}, the reduced space required by Algorithm~2 becomes critical for larger values of $q$.
To demonstrate this, we performed a sixth computation with the discriminant $D=-1000000013079299$, this time using $q\approx 2^{33220}$.
The total running time for Algorithm~2 was about 5800000 seconds (including root-finding), just a 20\% increase over the 256-bit computation, and the size of the CRT data was about 25 MB, yielding a total memory usage under 200 MB.
The 10000-digit prime $q$ and the coefficients of the constructed curve are too large to conveniently print here, but they are available at \url{http://math.mit.edu/~drew}.

By contrast, Algorithm~1, and the algorithm of \cite{EngeSutherland:CRTClassInvariants}, requires more than 20 GB of CRT data for this example, and this data needs to be updated for every prime $p\in S$.
This makes it infeasible to even attempt this computation with Algorithm~1 on our test platform, whereas Algorithm~2 was easily able to address this example.

\begin{table}
\begin{tabular}{@{}rrrrrrrrrr@{}}
alg&$n$&$b$&$|S|$&$\Tf$&$\Te$&$\Tb$&$\Tc$&$\Tp$&$\Tr$\\
   &   & (bits)  &  &(ms)&(ms)&(ms)&(ms)&(s)&(s)\\
\midrule\vspace{4pt}
-&       -& 21533401& 438700& 17500& 1580& 25000& 531& 19400000& 95600\\
1&   44399&  8315747& 170112& 12700& 1580&  9210& 531&  4120000&   237\\\vspace{4pt}
2&   44399&  8344202& 150662& 12700& 3160&  8350&   5&  4190000&   237\\
1&    3277& 11130011& 227504& 13700& 1580&  7180& 535&  5260000&     0\\\vspace{4pt}
2&    3277& 11518641& 235482& 14400& 3170&  2510&   0&  4780000&     0\\
\bottomrule
\end{tabular}
\\
\vspace{5pt}
\caption{CM computations with $|D|=10^{15}+13079299$ and $q\approx 2^{256}$.}\label{table:D15perf}
\end{table}

\medskip
Finally, we performed two record-setting computations, one with $h(D)>5\cdot 10^7$ and the other with $|D|>10^{16}$, again using the polynomial $\sqrt{H_D[A_{71}]}$.
First, we used the discriminant $D=-506112046263599$ with class number $h(D)=50666940$ to construct an Edwards curve of the form $x^2 + y^2 = 1 + cx^2y^2$, where
\footnotesize
\[
c = 3499565016101407566774046926671095877424725326083135202080143113943636512545,
\]
\normalsize
over the 256-bit prime field $\Fq$ with
\footnotesize
\[
q = 28948022309329048855892746252171986268338819619472424415843054443714437912893.
\]
\normalsize
The trace of this curve is
\footnotesize
\[
t = 340282366920938463463374607431768266146,
\]
\normalsize
and the group order $q+1-t$ is 4 times a prime.
This computation took approximately 200 days of CPU time (about 5 days elapsed time) using Algorithm 2, which in this case was faster than Algorithm 1.
\medskip

Next we used $D=-10000006055889179$ with class number $h(D)= 25459680$ to construct an elliptic curve
with Weierstrass equation $y^2 = x^3 -3x + c$, where
\footnotesize
\[
c = 15325252384887882227757421748102794318349518712709487389817905929239007568605,
\]
\normalsize
over the 256-bit prime field $\Fq$ with
\footnotesize
\[
q = 28948022309329048855892746252171992875431396939874100252456123922623314798263.
\]
\normalsize
This curve has trace
\footnotesize
\[
t = -340282366920938463463374607431768304979,
\]
\normalsize
and the group order is prime.
This computation took about 400 days of CPU time (under 8 days elapsed time) using Algorithm 1, which was faster than Algorithm 2 for this discriminant.

\section{Acknowledgements}
I am grateful to Andreas Enge and Fran\c{c}ois Morain for providing further details of the algorithms in \cite{EngeMorain:Decomposition,HanrotMorain:Solvability} and to David Harvey for his assistance with \texttt{zn\_poly}.
I would also like to sincerely thank the anonymous referee, whose careful reading and comprehensive feedback greatly improved the clarity and rigor of this article.

\bibliographystyle{amsplain}
\input{decomp.bbl}

\end{document}

%% file: decomp.bbl
\providecommand{\bysame}{\leavevmode\hbox to3em{\hrulefill}\thinspace}
\providecommand{\MR}{\relax\ifhmode\unskip\space\fi MR }
\providecommand{\MRhref}[2]{%
  \href{http://www.ams.org/mathscinet-getitem?mr=#1}{#2}
}
\providecommand{\href}[2]{#2}